\documentclass[final,nomarks]{dmtcs-episciences}

\usepackage[utf8]{inputenc}
\usepackage{subfigure}
\usepackage[round]{natbib}

\usepackage{hyperref}
\usepackage{amsmath}
\usepackage{amsthm}
\usepackage{enumerate}

\allowdisplaybreaks

\newcommand{\func}{\mathsf{func}}

\newcommand{\pMC}{\mathsf{pMC}}
\newcommand{\fpt}{\mathsf{fpt}}
\newcommand{\Cfactor}{\mathsf{pF}[\mathsf{C}]}
\newcommand{\Cfactorleq}{\mathsf{pF}_{\leq}[\mathsf{C}]}
\newcommand{\Cks}{\mathsf{pKS}[\mathsf{C}]}
\newcommand{\Cksleq}{\mathsf{pKS}_{\leq}[\mathsf{C}]}
\newcommand{\Csss}{\mathsf{pSSS}[\mathsf{C}]}
\newcommand{\Csssr}{\mathsf{pSSS}^{\neq}[\mathsf{C}]}
\newcommand{\Csssleq}{\mathsf{pSSS}_{\leq}[\mathsf{C}]}
\newcommand{\Csssrleq}{\mathsf{pSSS}_{\leq}^{\neq}[\mathsf{C}]}

\newcommand{\tmfactor}{\mathsf{pF}[\mathsf{T}_*]}

\newcommand{\lfpt}{\leq^{\fpt}}

\newcommand{\sgfactor}{\mathsf{pF}[\mathsf{S}_*]}
\newcommand{\sgfactorleq}{\mathsf{pF}_{\leq}[\mathsf{S}_*]}
\newcommand{\sgks}{\mathsf{pKS}[\mathsf{S}_*]}

\newcommand{\sgsss}{\mathsf{pSSS}[\mathsf{S}_*]}
\newcommand{\sgsssr}{\mathsf{pSSS}^{\neq}[\mathsf{S}_*]}

\newcommand{\cpgfactor}{\mathsf{pF}[\mathsf{CycPG}]}

\newcommand{\cpgksleq}{\mathsf{pKS}_{\leq}[\mathsf{CycPG}]}
\newcommand{\cpgsss}{\mathsf{pSSS}[\mathsf{CycPG}]}

\newcommand{\fcgks}{\mathsf{pKS}[\mathsf{FCycG}]}
\newcommand{\fcgksleq}{\mathsf{pKS}_{\leq}[\mathsf{FCycG}]}
\newcommand{\fcgsss}{\mathsf{pSSS}[\mathsf{FCycG}]}

\newcommand{\apgks}{\mathsf{pKS}[\mathsf{AbPG}]}
\newcommand{\apgsss}{\mathsf{pSSS}[\mathsf{AbPG}]}

\newcommand{\afgks}{\mathsf{pKS}[\mathsf{FAbG}]}
\newcommand{\afgsss}{\mathsf{pSSS}[\mathsf{FAbG}]}

\newcommand{\Zsss}{\mathsf{pSSS}[\integers]}

\newcommand{\Zksleq}{\mathsf{pKS}_{\leq}[\integers]}
\newcommand{\Zsssr}{\mathsf{pSSS}^{\neq}[\integers]}
\newcommand{\Zsssleq}{\mathsf{pSSS}_{\leq}[\integers]}

\newcommand{\Znks}{\mathsf{pKS}[\integers^*]}
\newcommand{\Znksleq}{\mathsf{pKS}_{\leq}[\integers^*]}
\newcommand{\Znsss}{\mathsf{pSSS}[\integers^*]}
\newcommand{\Znsssr}{\mathsf{pSSS}^{\neq}[\integers^*]}

\newcommand{\Nks}{\mathsf{pKS}[\naturals]}
\newcommand{\Nksleq}{\mathsf{pKS}_{\leq}[\naturals]}

\newcommand{\Nsss}{\mathsf{pSSS}[\naturals]}
\newcommand{\Nsssr}{\mathsf{pSSS}^{\neq}[\naturals]}
\newcommand{\Nsssleq}{\mathsf{pSSS}_{\leq}[\naturals]}
\newcommand{\Nsssleqr}{\mathsf{pSSS}_{\leq}^{\neq}[\naturals]}

\newcommand{\Nnsss}{\mathsf{pSSS}[\naturals^*]}
\newcommand{\Nnsssr}{\mathsf{pSSS}^{\neq}[\naturals^*]}

\newcommand{\Nnks}{\mathsf{pKS}[\naturals^*]}

\theoremstyle{plain}
\newtheorem{theorem}{Theorem}[section]
\newtheorem{problem}[theorem]{Problem}
\newtheorem{definition}[theorem]{Definition}
\newtheorem{remark}[theorem]{Remark}
\newtheorem{lemma}[theorem]{Lemma}
\newtheorem{corollary}[theorem]{Corollary}
\newtheorem{newclaim}[theorem]{Claim}

\DeclareMathOperator\ord{ord}
\DeclareMathOperator\lcm{lcm}
\DeclareMathOperator\poly{poly}

\newcommand{\llbracket}{[\![}
\newcommand{\rrbracket}{]\!]}

\author{Markus Lohrey \and Andreas Rosowski\thanks{Supported by the DFG research project LO 748/12-2.}}
\title{Parameterized complexity of factorization problems}
\affiliation{University of Siegen, Department ETI, Siegen, Germany}

\publicationdata{vol. 27:3}{2025}{8}{10.46298/dmtcs.13087}{2024-02-20; 2024-02-20; 2025-02-14}{2025-08-18}

\begin{document}

\maketitle
\begin{abstract}
We study the parameterized complexity of the following factorization problem:
given elements  $\pi,\pi_1, \ldots, \pi_m$ of a monoid and a parameter $k$,
can $\pi$ be written as the product of at most (or exactly) $k$ elements from $\pi_1, \ldots, \pi_m$?
Also, two restricted variants of this problem (subset sum and knapsack) are considered.
For these problems we show several new upper complexity bounds and fpt-equivalences 
for various classes of monoids.
Finally, some new upper complexity bounds for variants of the parameterized change-making problems are shown.
\end{abstract}

\section{Introduction}

This paper deals with the parameterized complexity of factorization problems in monoids. In the most generic setting,
we have a fixed class $\mathsf{C}$ of monoids. The input consists of a monoid $M \in \mathsf{C}$, elements $a, a_1,\dots,a_m \in M$ and
a positive integer $k$. The question is whether $a$ can be written as a product of at most $k$ elements from the list $a_1,\dots,a_m$ (every
$a_i$ may occur several times in this product). In the parameterized setting, $k$ is the parameter. This problem has been studied most intensively
for the case, where $\mathsf{C}$ consists of the class of all symmetric groups $S_n$ ($S_n$ is the group of all permutations on the sets $\{1,2,\ldots,n\}$). 
The non-parameterized variant of this problem has been introduced by Even and Goldreich under the name MGS (for minimum generator sequence) \cite{EvenG81}, where the problem
was shown to be ${\sf NP}$-complete if $k$ is given in unary notation. This results even holds for the case, when $\mathsf{C}$ is the class of abelian groups
$\integers_2^n$ for $n \geq 1$. For the case, where the factorization length $k$ is given in binary notation, Jerrum \cite{Jerrum85} proved that MGS is {\sf PSPACE}-complete.

Let us now come to the parameterized factorization problem for permutations, where the parameter is the factorization length $k$. We denote this problem
by $\sgfactor$ (the $*$ indicates that the degree $n$ of the symmetric group $S_n$ is part of the input).
It was shown by \cite[Chapter 11]{DoFe99}\footnote{In \cite{DoFe99} the problem $\sgfactor$ is called
\textsc{Permutation Group Factorization}.}
 that $\sgfactor$ is $\mathsf{W}[1]$-hard and hence unlikely to be fixed parameter tractable.
On the other hand, no good upper bound is known.
Interestingly, the parameterized factorization problem for transformation monoids $T_n$ 
($T_n$ is the monoid consisting of all mappings on the finite set $\{1,2,\ldots,n\}$), which is denoted by $\tmfactor$, is $\mathsf{W}[2]$-hard \cite{caichendowneyfellows}.
It was shown in \cite{fernaubruchertseifer} that $\tmfactor$ is equivalent under fixed parameter reductions
to the problem whether a given deterministic automaton has a synchronizing word of length at
most $k$, where $k$ is the parameter.  The latter problem belongs to the intersection of the parameterized classes
$\mathsf{A}[2]$, $\mathsf{W}[\mathsf{P}]$, and {\sf WNL} (see Section~\ref{sec-para-ct} for a definition). In particular, $\tmfactor$
also belongs to the classes $\mathsf{A}[2]$, $\mathsf{W}[\mathsf{P}]$, and {\sf WNL}.  Our first result improves on the membership in $\mathsf{A}[2]$
by showing that $\tmfactor$ belongs to the class $\mathsf{W}^\textsf{func}[2] \subseteq  \mathsf{A}[2]$.
The classes $\mathsf{W}^\textsf{func}[2]$ and $\mathsf{W}[2]$ are both defined by parameterized model-checking problems for certain fragments of first-order
logic (see Section~\ref{sec-para-ct} for details);
the only difference is that $\mathsf{W}[2]$ restricts to formulas without function symbols, whereas $\mathsf{W}^\textsf{func}[2]$ allows also function symbols.
 
 We then proceed with showing that $\sgfactor$ is equivalent with respect to fixed parameter reduction to two apparently simpler problems: the parameterized knapsack problem
 for symmetric groups $\sgks$ and the parameterized subset sum problem for symmetric groups $\sgsss$. The problem $\sgks$ is a restricted version of $\sgfactor$ in the sense
 that only factorizations of the form $a = a_1^{x_1} a_1^{x_2} \cdots a_m^{x_m}$ with $x_1, \ldots, x_m \in \naturals$ and $\sum_{i=1}^m x_i \leq k$ are allowed.
 For $\sgsss$ one additionally requires $x_1, \ldots, x_m \in \{0,1\}$. In the classical (unparameterized) setting, these problems are again {\sf NP}-complete \cite{LohreyRZ22}.
 
 When restricting the problems $\sgks$ and $\sgsss$ to cyclic permutation groups we will show that the resulting problems are equivalent to the parameterized variant
 of the classical subset sum problem for integers. This problem is known to be in $\mathsf{W}[3]$ \cite{bussislam} and $\mathsf{W}[1]$-hard
  \cite{downeyfellows}. Hence, the same complexity bounds hold for  $\sgks$ and $\sgsss$ when restricted to cyclic permutation groups.
  
 In the last section of the paper, we deal with the change-making problems introduced in \cite{goebbels}, which are variants of the classical
 knapsack and subset sum problems for integers, precise definitions can be found in Section~\ref{sec-change-making}.
It was shown in \cite{goebbels} that these change-making problems are $\mathsf{W}[1]$-hard and contained in $\mathsf{XP}$ (see Section~\ref{sec-para-ct} for a definition). 
We observe that by using the above mentioned result from \cite{bussislam} one can improve the upper bound to $\mathsf{W}[3]$. More interestingly,
we also consider the approximative change making problems from \cite{goebbels} that involve an objective function $f(x,y) = ax+by$ that has to be
minimized. The approximative change making problems are in
  $\mathsf{W}[3]$ (if $b > 0$ or $a=0$) or \textsf{para-NP}-complete (if $b= 0$ and $a > 0$)  (see Section~\ref{sec-para-ct} for a definition of \textsf{para-NP}).
 
 \paragraph{Related work.} The knapsack problem and subset problem have been also studied intensively for finitely generated infinite groups;
see \cite{BKLMNRS20} for a survey. In \cite{GorKou95} a variant of the factorization problem for transformation monoids has been considered, where
it is asked whether for given transformations $a_1, \ldots, a_n \in T_n$ and $k \geq 1$
 there is an idempotent transformation of a certain fixed rank (which is the size of the image) that can be written as a product
of at most $k$ transformations from the list $a_1, \ldots, a_n$.
  
\section{General notations}

We write $\poly(n)$ for an arbitrary polynomial in $n$. The composition $fg$ of two functions $f,g : A \to A$ is evaluated from the left
to right, i.e., $(fg)(a) = g(f(a))$. Therefore we also write $a f$ for $f(a)$. Then we have $a(fg) = afg$. In order to distinguish better between
the argument $a$ and the function $f$ we sometimes also write $a \cdot f$ for $af$.

For a positive integer $N \in \naturals$ we denote by $\#(N)$ the number of bits of $N$.
For integers $i \leq j$ we write $[i,j]$ for the interval $\{i,i+1, \ldots, j\}$. Let $T_n$ be the set of all mappings on $[1,n]$
and $S_n$ be the set of all permutations on $[1,n]$. With respect to composition of functions, $T_n$ is a monoid (the \emph{transformation monoid}
on $n$ elements) and $S_n$ is a group (the \emph{symmetric group} on $n$ elements).
A permutation group is a subgroup of  $S_n$ for some $n$.\footnote{By Cayley's theorem, every finite group is isomorphic to some subgroup of $S_n$. When we talk about
a permutation group $G$, we mean the abstract group $G$ together with its action on the elements of $[1,n]$.}
We will use standard notations for permutation groups; see e.g., \cite{seress03}. Permutations will be often written 
by their decomposition into disjoint cycles, called the \emph{disjoint cycle decomposition}. 
A cycle of length two is a \emph{transposition}. 
For a permutation $\pi$ we denote by $\ord(\pi)$ the order of $\pi$, i.e., the smallest $k \geq 1$ such that $\pi^k$ is
the identity permutation.

 Let $G$ be a finite group. For a subset $A \subseteq G$ we denote with $\langle A \rangle$ the
subgroup of $G$ generated by the elements from $A$ (i.e., the closure of $A$ under the group multiplication).
If $\langle A \rangle = G$ then $A$ is called a generating set of $G$.
For $k \geq 0$ we write $A^{\leq k}$ for the set of all products $a_1 a_2 \cdots a_l \in G$
with $l \leq k$ and $a_1, \ldots, a_l \in A$. 
For an element $g \in \langle A \rangle$ we denote with $|g|_A$ (the $A$-length of $g$)
the smallest integer $k$ such that $g \in A^{\leq k}$.

Several times we use the well-known Chinese remainder theorem in the following version:

\begin{theorem}[\mbox{see e.g.~\cite[Chapter~3, \S 4]{ireland}}] \label{CRT}
Let $n_1, n_2, \ldots, n_l \geq 2$ be pairwise coprime integers and let $N = \prod_{1 \le i \le l} n_i$.
 Then there is a ring isomorphism
\[ h : \prod_{1 \le i \le l} \integers_{n_i} \to \integers_N .\]
Moreover, given the binary representations of $n_1, \ldots, n_l \in \naturals$ and 
$x_i \in \integers_{n_i}$ for $1 \le i \le l$, one can compute in polynomial time the binary representation of 
$h(x_1, \ldots, x_l) \in \integers_N$.\footnote{This follows from the standard proof of the Chinese remainder theorem; see \cite[proof of Theorem 1 on p.~34]{ireland}.}
\end{theorem}
We will only use the fact that $h$ is an isomorphism between the additive abelian groups 
$\prod_{1 \le i \le l} \integers_{n_i} $ and  $\integers_N$.

\section{Parameterized complexity theory}  \label{sec-para-ct}

An excellent introduction into parameterized complexity theory is the monograph \cite{flumgrohe}.
A \emph{parameterized problem} is a subset $L \subseteq \Sigma^* \times \naturals$ for a finite alphabet $\Sigma$.
For a pair $(u, k) \in \Sigma^* \times \naturals$ we call $k$ the parameter.
For a fixed number $d \in \naturals$, the $d^{\text{th}}$ slice of $L$ is $L_d = \{ (u,d) \in L : u \in \Sigma^* \}$.
The class {\sf FPT} of \emph{fixed parameter tractable problems} is the class of all parameterized problems $L \subseteq \Sigma^* \times \naturals$ for which there is an algorithm
deciding $L$ and running on input $(u,k)$ in time $poly(|u|) \cdot f(k)$ for an arbitrary computable function $f : \naturals \to \naturals$.
The class \textsf{XP} is the class of parameterized problems that can be solved by a \emph{deterministic} algorithm in time 
$|u|^{f(k)} + f(k)$ for an arbitrary computable function $f : \naturals \to \naturals$ \cite[Definition 2.22]{flumgrohe}.
The class \textsf{para-NP} is the class of parameterized problems that can be solved by a \emph{nondeterministic} algorithm in time 
$poly(|u|) \cdot f(k)$ for an arbitrary computable function $f : \naturals \to \naturals$ \cite[Definition 2.10]{flumgrohe}.
The class $\mathsf{W}[\mathsf{P}]$ is the class of parameterized problems that can be solved by a \emph{nondeterministic} algorithm in time 
$poly(|u|) \cdot f(k)$ for an arbitrary computable function $f : \naturals \to \naturals$ with the restriction that at most
$\log(|u|) \cdot h(k)$ steps are allowed to be nondeterministic ones for an arbitrary computable function $h : \naturals \to \naturals$ \cite[Definition~3.1]{flumgrohe}. 
By \cite[Proposition~3.2]{flumgrohe} we have
\begin{displaymath}
{\sf FPT} \subseteq \mathsf{W}[\mathsf{P}] \subseteq \textsf{XP} \cap \textsf{para-NP}.
\end{displaymath}
An \emph{fpt-reduction} from a parameterized problem $L_1 \subseteq \Sigma_1^* \times \naturals$ to a parameterized 
problem $L_2 \subseteq \Sigma_2^* \times \naturals$ is a computable function $r : \Sigma_1^* \times \naturals \to \Sigma_2^* \times \naturals$ such that
the following hold for all $u \in \Sigma_1^*$ and $k \in \naturals$:
\begin{itemize}
\item there are functions $g,h$ such that $r(u,k) = (g(u,k), h(k))$,
\item $r(u,k)$ can be computed in time $poly(|u|)  f(k)$ for a computable function $f : \naturals \to \naturals$, 
\item $(u,k) \in L_1$ if and only if $r(u,k) \in L_2$.
\end{itemize}
We write $L_1 \lfpt L_2$ if there is an \emph{fpt-reduction} from $L_1$ to $L_2$. We write
$L_1 \equiv^{\fpt} L_2$ for $L_1 \lfpt L_2 \lfpt L_1$. Moreover, for the parameterized problem $L$
we define its fpt-closure $[L]^{\fpt}$ as the class of all parameterized problems $K$ with $K \lfpt L$.

In \cite{guillemot} the class {\sf WNL} is defined as $[\textsf{NTMC}]^{\fpt}$, where NTMC (for \emph{nondeterministic Turing machine computation})
is the following parameterized problem:
\begin{problem}[{\sf NTMC}]\label{prob-NTMC}~\\
Input: a nondeterministic Turing machine $M$, a unary encoded $q \in \naturals$ and $k \in \naturals$ (the parameter)\\
Question: Does $M$ accept the empty string in $q$ steps by examining at most $k$ cells?
\end{problem}
Of particular importance in parameterized complexity theory are model-checking problems for fragments of first-order logic.
We rely on the definition of a structure given in \cite{flumgrohe}. 
We fix a countably infinite set of \emph{variables}.
A \emph{signature} $\sigma$ is a finite set of \emph{relational symbols} and \emph{function symbols}, 
where every relation symbol $r \in \sigma$ has an \emph{arity} 
$\alpha(r) \in \naturals \setminus \{0\}$ and every function symbol $f \in \mathcal{R}$ has an arity
$\alpha(f) \in \naturals$. Function symbols of arity $0$ are also called \emph{constant symbols}.
A \emph{$\sigma$-structure} (or just structure, if the signature is not important) 
is a tuple $\mathcal{A} = (A, (s^{\mathcal{A}})_{s \in \sigma})$, where $A$ is a non-empty set (the universe 
of $\mathcal{A}$). For every relation symbol $r \in \sigma$, 
$r^{\mathcal{A}}\subseteq A^{\alpha(r)}$ is an $\alpha(r)$-ary relation on  the universe.
For every function symbol $f \in \sigma$, $f^\sigma : A^{\alpha(f)} \to A$ is an $\alpha(f)$-ary function on $A$.
A \emph{relational signature} only contains relational symbols and a \emph{relational structure} is a $\sigma$-structure for 
a relational signature $\sigma$. 
 
\emph{Terms} over the signature $\sigma$ are built in the usual way: every variable and constant symbol is a term and if 
$t_1, \ldots, t_n$ are terms and $f$ is a function symbol with $\alpha(f) = n$ then also $f(t_1, \ldots, t_n)$ is a term.
The size $|t|$ of a term $t$ is inductively defined by $|x| = |a|=1$ for a variable $x$ and a constant symbol $a$ and 
$|f(t_1, \ldots, t_n)| = 1 + \sum_{i=1}^n |t_i|$. Note that if the signature $\sigma$ is relational then the only terms are variables.

\emph{First-order formulas} over the signature $\sigma$ are built up from atomic formulas of the form $r(t_1, \ldots, t_n)$ (where
$r \in \sigma$ is a relational symbol, $\alpha(r)=n$, and $t_1, \ldots, t_n$ are terms) and $t_1=t_2$ (for terms $t_1, t_2$) using boolean operators
($\neg$, $\wedge$ and $\vee$) and quantifiers ($\exists x$ and $\forall x$ for a variable $x$).
A \emph{relational first-order formula} is a first-order formula $F$ such that every term that appears in $F$ is a variable.
A (first-order) \emph{sentence} is a first-order formula without free variables, i.e., every occurrence of a variable $x$ that appears in an atomic formula
is within the scope of a quantifier $\exists x$ or $\forall x$.
For a sentence $F$, we write $\mathcal{A} \models F$ if the formula $F$ is true (in the usual sense) in the structure $\mathcal{A}$.
For a first-order formula $F$ we define the length $|F|$ of the formula inductively:
\begin{itemize}
\item $|t_1=t_2|=1+|t_1|+|t_2|$ for terms $t_1, t_2$,
\item $|r(t_1, \dots, t_n) = 1 + \sum_{i=1}^n |t_i|$ for terms $t_1, \ldots, t_n$ and a relational symbol $r$,
\item $|F \land G| = |F \lor G|  =  |F| + |G|+1$,
\item $|\neg F| = |\exists x F| = |\forall x F| = |F|+1$.
\end{itemize}
For a class of first-order sentences $\mathsf{C}$ we define 
the parameterized problem $\pMC(\mathsf{C})$ as follows:
\begin{problem}[parameterized model-checking problem for $\mathsf{C}$, $\pMC(\mathsf{C})$ for short]~\\
Input: a sentence $F \in \mathsf{C}$ and a structure $\mathcal{A}$\\
Parameter: the length $|F|$ of the formula $F$\\
Question: Does $\mathcal{A} \models F$ hold?
\end{problem}
We now define several well-known parameterized complexity classes by taking for $\mathsf{C}$ certain 
fragments of first-order logic.
For $l \geq 1$ we define $\Sigma_l$ as the class of relational first-order sentences of the form
$\exists \overline{x}_1 \forall \overline{x}_2 \cdots Q_l \overline{x}_l \psi$,
where $\psi$ is a relational quantifier-free formula,
$Q_l = \exists$ if $l$ is odd, $Q_l = \forall$ if $l$ is even, and every
$\overline{x}_i$ is a tuple of variables of arbitrary length. 
If we require in addition that every tuple $\overline{x}_i$ with $2 \leq i \leq l$ has at most $u$ variables,
we obtain the class $\Sigma_{l,u}$. If we allow arbitrary terms in these
classes, then we obtain $\Sigma^\func_{l}$ and $\Sigma^\func_{l,u}$, respectively.

We then define the following parameterized complexity classes, where $l \geq 1$:
\begin{alignat}{2}
\textsf{W}[l] & \ = \ \big[\pMC(\Sigma_{l,1})\big]^{\textsf{fpt}} & \quad\quad &  \text{\cite[Theorem~7.22]{flumgrohe}} \label{def-W} \\
\textsf{A}[l]  & \ = \ [\pMC(\Sigma_l)]^{\textsf{fpt}}   &  & \text{\cite[Definition~5.7]{flumgrohe}} \nonumber \\
\textsf{W}^\textsf{func}[l] & \ = \ \big[\pMC(\Sigma_{l,1}^\func)\big]^{\textsf{fpt}} & &  \text{\cite[Definition~28]{ChenFG05}}  \nonumber \\
\textsf{A}^\textsf{func}[l]  & \ = \  [\pMC(\Sigma^\func_l)]^{\textsf{fpt}} .   &  &   \nonumber
\end{alignat}
The following facts are known about these classes:
\begin{itemize}
\item $\textsf{W}[1]  = \textsf{A}[1]$ (this follows directly from our definition).
\item $\textsf{W}[l]  = \big[\pMC(\Sigma_{l,u})\big]^{\textsf{fpt}}$ \cite[Theorem~7.22]{flumgrohe} and $\textsf{W}^\func[l]  = \big[\pMC(\Sigma^\func_{l,u})\big]^{\textsf{fpt}}$ \cite[Theorem~29]{ChenFG05} for every
$u \geq 1$  
\item $\textsf{W}[l] \subseteq \textsf{W}^\textsf{func}[l] \subseteq \textsf{A}^\textsf{func}[l] = \textsf{A}[l]$ for all $l \geq 1$
(only the last equality is nontrivial; it is stated in \cite[p.~201]{flumgrohe}
\item $\displaystyle \bigcup_{l \geq 1} \textsf{W}[l] \subseteq {\sf WNL} \subseteq \textsf{XP} \cap \textsf{para-NP}$ \cite[Theorem~1]{fernaubruchertseifer}
\item  $\displaystyle \bigcup_{l \geq 1} \textsf{W}^\textsf{func}[l] \subseteq \textsf{W}[\textsf{P}]$ \cite[Exercise~8.58]{flumgrohe}
\end{itemize}
Finally, let us remark that if in the problem NTMC (Problem~\ref{prob-NTMC}) one takes 
$q$ instead of $k$ as the parameter, then the fpt-closure of the resulting problem is $\textsf{W}[1]$;
see \cite{guillemot}.

\section{Parameterized complexity of factorization problems}

\subsection{Generic variants of factorization problems} \label{sec-generic}

In the following the parameter will be always the natural number $k$.
It therefore makes no difference whether the parameter $k$ is given in binary or unary representation, 
since it enters the running time with an arbitrary function $f(k)$.

Let $\mathsf{C}$ be a class of monoids. Every monoid $M \in \mathsf{C}$ must have a finite representation that we define later for the 
classes considered in this paper. Also elements of $M$ must have a finite representation.
The parameterized factorization problem for the class $\mathsf{C}$ is the following problem:

\begin{problem}[parameterized factorization problem for the class $\mathsf{C}$, $\Cfactor$ for short]~\\
Input: a monoid $M \in \mathsf{C}$, elements $a, a_1,\dots,a_m \in M$ and $k \in \naturals$ (the parameter)\\
Question: Is $a \in \{a_1,\dots,a_m\}^{k}$?
\end{problem}
We also consider two variants of $\Cfactor$ by restricting the order in which the elements $a_1,\dots,a_m$
may occur.

\begin{problem}[parameterized knapsack for $\mathsf{C}$, $\Cks$ for short]\label{pknapsack}~\\
Input: a monoid $M \in \mathsf{C}$, elements $a, a_1,\dots,a_m \in M$ and $k \in \naturals$ (the parameter)\\
Question: Are there $x_1,\dots,x_m \in \naturals$ with $\sum_{i=1}^m x_i = k$ and $a = a_1^{x_1} \cdots a_m^{x_m}$?
\end{problem}

\begin{problem}[parameterized subset sum for $\mathsf{C}$, $\Csss$ for short]\label{psubsetsum}~\\
Input: a monoid $M \in \mathsf{C}$, elements $a, a_1,\dots,a_m \in M$ and $k \in \naturals$ (the parameter)\\
Question: Are there $x_1,\dots,x_m \in \{0,1\}$ with $\sum_{i=1}^m x_i = k$ and $a = a_1^{x_1} \cdots a_m^{x_m}$?
\end{problem}
Note that when $\mathsf{C}$ only contains commutative monoids, the problems $\Cfactor$ and $\Cks$ are obviously equivalent with respect to fpt-reductions.
In this case, we only consider the problems $\Cks$ and $\Csss$. 
If $\mathsf{C} = \{M\}$ consists of only one monoid, we write
$\mathsf{pSSS}[M]$ instead of $\mathsf{pSSS}[\{M\}]$ and analogously for the other problems.

In this paper, the focus is on permutation groups and abelian groups.

\begin{definition}
\label{def-monoid-classes} \rm
We consider the following cases for the class $\mathsf{C}$:
\begin{itemize}
\item $\mathsf{T}_* := \{ T_n : n \geq 1\}$ is the class of all transformation monoids.  The transformation monoid $T_n$ is simply represented by the number $n$ 
in \emph{unary} representation and a function $f \in T_n$ is represented by the list $f(1), f(2), \ldots, f(n)$. The problem $\tmfactor$ has been considered in \cite{caichendowneyfellows}.
\item $\mathsf{S}_* := \{ S_n : n \geq 1\}$ is the class of all symmetric groups. For this case, we inherit the above input representation for $\mathsf{T}_*$.
\item $\mathsf{AbPG}$ is the class of abelian permutation groups. An abelian permutation group is represented by a list of pairwise commuting generators $\pi_1, \ldots, \pi_m \in S_n$ (for some $n$).
Elements of $\langle \pi_1, \ldots, \pi_m \rangle$ are simply represented by permutations (by \cite{FurstHL80} one can check in polynomial time, whether a given permutation belongs to
$\langle \pi_1, \ldots, \pi_m \rangle$).
\item $\mathsf{CycPG}$ is the class of cyclic permutation groups. We use the same input representation as for $\mathsf{AbPG}$.
In Appendix~\ref{appendix-cyclic}
we will show that given a list of permutations $\pi_1, \ldots, \pi_m$ one can check in polynomial time, whether the permutation group $\langle \pi_1, \ldots, \pi_m \rangle$
is cyclic. In the positive case one can compute in polynomial time a permutation $\pi$ with $\langle \pi\rangle = \langle \pi_1, \ldots, \pi_m \rangle$.
\item the singleton classes $\{\integers\}$ and $\{\naturals\}$: in these cases, there is of course no reason to choose an input representation for the monoid. 
Elements of $\integers$ or $\naturals$ are represented in binary representation. 
\item $\integers^* := \{\integers^n : n \geq 1 \}$ 
and $\naturals^* := \{\naturals^n : n \geq 1 \}$, where $\integers^n$ (respectively, $\naturals^n$) is represented by the unary encoding of $n$ and group elements
are represented by $n$-tuples of binary encodings.
\item $\mathsf{FAbG}$ is the class of finite abelian groups, where the finite abelian group $G =  \prod_{i=1}^d \integers_{n_i}$
is represented by the tuple $(n_1, n_2, \ldots, n_d)$ with each $n_i$ given in binary representation. An element of $G$ is given by a tuple $(a_1, a_2, \ldots, a_d)$ with each $a_i \in [0,n_i-1]$
given in binary representation. 
\item  $\mathsf{FCycG} := \{ \integers_n : n \geq 2 \}$ is the class of finite cyclic groups (with $n$ again given in binary notation).
\end{itemize}
The classes except for $\mathsf{C} = \mathsf{T}_*$ and $\mathsf{C} = \mathsf{S}_*$ are called the \emph{commutative classes} in the rest of the paper.
For those classes we do not consider $\Cfactor$ since it is equivalent to $\Cks$ with respect to fpt-reductions.
\end{definition}

\begin{remark}
Note that as abstract classes of groups, $\mathsf{CycPG}$ and $\mathsf{FCycG}$ (respectively, $\mathsf{AbPG}$ and $\mathsf{FAbG}$) are the same, but there is a difference
in the succinctness of descriptions.
In general, if  $\gamma_1, \ldots, \gamma_d$ are the pairwise disjoint cycles of a permutation $\pi \in S_n$, and $\ell_i \in [1,n]$ is the length of $\gamma_i$, then
$\ord(\pi) = \lcm(\ell_1, \ldots, \ell_d)$. In particular, every prime power $p^e$ that divides $\ord(\pi)$ is bounded by $n$. Therefore the cyclic permutation group $\langle \pi \rangle \leq S_n$
is isomorphic to  $\integers_m$, where all prime powers in $m$ are bounded by $n$ (recall that the latter is given in unary representation in our setting).
Hence, if for instance, $\integers_{2^m} \cong \langle \pi \rangle$ for $\pi \in S_n$, 
 then $n \geq 2^m$ and the encoding of $\pi$ (a list of $n$ numbers from $[1,n]$)
needs at least $2^m \cdot m$ bits.
\end{remark}

\begin{remark}
For an abelian permutation group $\langle \pi_1, \ldots, \pi_m \rangle$ we represent an element $\pi \in \langle \pi_1, \ldots, \pi_m \rangle$
as a permutation. Alternatively, we could represent $\pi$ by binary encoded exponents $k_1, \ldots, k_m \in \naturals$ such that
$\pi = \pi_1^{k_1} \cdots \pi_m^{k_m}$. These two representations are equivalent in the sense that one can change in polynomial time between the 
two representations. One direction is clear: given binary encoded numbers $k_1, \ldots, k_m$ one can compute by iterated squaring the 
permutation $\pi_1^{k_1} \cdots \pi_m^{k_m}$. On the other hand,
given a permutation $\pi \in \langle  \pi_1, \ldots, \pi_m \rangle$, the 
algorithm from \cite{FurstHL80} allows to compute in polynomial time a straight-line program over $\pi_1, \ldots, \pi_m$
for $\pi$. This is a context-free grammar that produces a single word $w$ over the letters $\pi_1, \ldots, \pi_m$ that evaluates to $\pi$.
One can then easily compute in polynomial time the number of occurrences of every $\pi_i$ in $w$.
\end{remark}
In some papers, one finds the alternative of $\Cfactor$, where the condition $a \in \{a_1,\dots,a_m\}^{k}$ is replaced by
$a \in \{a_1,\dots,a_m\}^{\leq k}$. Similarly, in $\Cks$ and $\Csss$ one could replace the condition $\sum_{i=1}^m x_i = k$ by
$\sum_{i=1}^m x_i \leq k$. Let us denote these problems with $\Cfactorleq$, $\Cksleq$, and $\Csssleq$, respectively.

For the subset sum problems one also finds the variants, where the $a_i$ are assumed to be pairwise
different. These variants will be denoted with $\Csssr$ and $\Csssrleq$. In the case where $\mathsf{C}$ consists of commutative monoids, one can define these problems also as follows:
given a monoid $M \in \mathsf{C}$, a finite subset $A \subseteq M$, an element $a \in M$
and the parameter $k$, the question is whether there is a subset $S \subseteq A$ with $|S|=k$ (respectively, $|S| \leq k$) and $\sum_{b \in A} b = a$.
The problem $\Zsssr$ was studied in \cite{bussislam}, whereas $\Zsss$ is the variant studied in \cite{downeyfellows}.
We will later show that the above variants of subset sum are all equivalent with respect to fpt-reductions for the classes $\mathsf{C}$ from Definition~\ref{def-monoid-classes}.
By replicating group elements in the input list one easily shows the following:
\begin{lemma} \label{lemma-simple}
For every class $\mathsf{C}$ we have $\Cfactor \lfpt \Cks \lfpt \Csss$ and
$\Cfactorleq \lfpt \Cksleq \lfpt \Csssleq$.
\end{lemma}
\begin{proof}
We only show $\Cfactor \lfpt \Cks \lfpt \Csss$; the second statement can be shown in the same way.
For $\Cfactor \lfpt \Cks$ consider input data $M \in \mathsf{C}$, $a, a_1,\dots, a_m \in M$ and $k \in \naturals$ for $\Cfactor$.
We clearly have $a \in \{a_1,\dots,a_m\}^{k}$ if and only if there exist $x_{i,j} \in \naturals$ ($i \in [1,k]$, $j \in [1,m]$) such that
\begin{equation} \label{eq-i}
a = \prod_{i=1}^k \prod_{j=1}^m  a_j^{x_{i,j}} \text{ and } \sum_{i=1}^k \sum_{j=1}^m x_{i,j} = k.
\end{equation}
To show that $\Cks \lfpt \Csss$ note that there exist $x_1, \ldots, x_m \in \naturals$ such that
\[
a = a_1^{x_1}  a_2^{x_2}  \cdots  a_m^{x_m}  \text{ and } \sum_{i=1}^m  x_{i} = k
\]
if and only if there exist $x_{i,j} \in \{0,1\}$ ($i \in [1,m]$, $j \in [1,k]$) such that
\[
a = a_1^{x_{1,1}} \cdots a_1^{x_{1,k}} a_2^{x_{2,1}} \cdots a_2^{x_{2,k}} \cdots a_m^{x_{m,1}} \cdots a_m^{x_{m,k}}  \text{ and } \sum_{i=1}^m \sum_{j=1}^k x_{i,j} = k.
\]
\end{proof}

\subsection{Factorization in transformation monoids} 

In \cite{caichendowneyfellows} it was shown that $\tmfactor$ is $\mathsf{W}[2]$-hard.
The complexity of $\tmfactor$ was related in \cite{fernaubruchertseifer} to another parameterized problem.
Consider a deterministic finite automaton $A = (Q, \Sigma, q_0, \delta, F)$ (DFA for short), where
$Q$ is the finite set of states, $\Sigma$ is the finite input alphabet, $q_0 \in Q$ is the initial state, $F \subseteq Q$ is the set
of final states and $\delta: Q \times \Sigma \to Q$ is the transition function.
A \emph{synchronizing word} $w \in \Sigma^*$ for $A$ has the property that there exists a state 
$q \in Q$ such that for every state $p \in Q$ we have $\hat{\delta}(p,w)=q$, where
$\hat{\delta}: Q \times \Sigma^* \to Q$ is the extension of $\delta$ to words (such a synchronizing word does not necessarily exist). 
The following problem has been introduced in \cite{FernauHV15}; see also \cite{fernaubruchertseifer}.

\begin{problem}[$\mathsf{p\text{-}DFA\text{-}SW}$ \cite{fernaubruchertseifer}]\label{pdfasw}~\\
Input: DFA $A$, $k \in \naturals$ (the parameter)\\
Question: Is there a synchronizing word $w$ for $A$ with $|w| \leq k$?
\end{problem}
It is known that {\sf p-DFA-SW} $\equiv^{\fpt} \tmfactor$ \cite{fernaubruchertseifer} and that
$\mathsf{p\text{-}DFA\text{-}SW}$ belongs to the intersection of the classes $\mathsf{A}[2]$, $\mathsf{W[P]}$, and $\mathsf{WNL}$ \cite{FernauHV15}.
The complexity class $\mathsf{W}[\textsc{sync}]$ was defined in  \cite{fernaubruchertseifer} as
$\mathsf{W}[\textsc{sync}] = [\mathsf{p\text{-}DFA\text{-}SW}]^{\fpt}$.
Therefore, $\tmfactor$ is $\mathsf{W}[\textsc{sync}]$-complete and belongs to the intersection of the classes $\mathsf{A}[2]$, $\mathsf{W[P]}$, and $\mathsf{WNL}$.
Note that $\mathsf{W}^{\func}[2] \subseteq \mathsf{A}^{\func}[2] = \mathsf{A}[2]$. Therefore, the following result  improves the membership of
 $\tmfactor$  in $\mathsf{A}[2]$ (and makes the statement $\tmfactor \in \mathsf{W[P]}$ obsolete since $\mathsf{W}^{\func}[2] \subseteq\mathsf{W[P]}$).
\begin{theorem}  \label{thm-tmfactor}
$\tmfactor$ belongs to 
$\mathsf{W}^{\func}[2]$. 
\end{theorem}
\begin{proof}
We present an fpt-reduction from $\tmfactor$ to $\pMC(\Sigma_{2,1}^{\func})$. Consider an input instance $(f, B, k)$ for $\tmfactor$, 
where $f \in T_n$, $B \subseteq T_n$ and $k \in \naturals$. W.l.o.g. we can assume that $B$ contains the identity function $\mathrm{id}$ with $\mathrm{id}(a)=a$ for all $a \in [1,n]$.
Then, $(f, B, k)$ is a positive instance of $\tmfactor$ if and only if there are $f_1,\dots,f_k \in B$ such that $af = af_{1}  \cdots f_{k}$
 for all $a \in [1,n]$. Let $\sigma = \{P,Q,h\}$ be a signature in which $P$ is a unary relation symbol, $Q$ is a binary relation symbol, and $h$ is a binary function symbol. We define the $\sigma$-structure $\mathcal{A}$ as follows: 
 \begin{itemize}
\item $A = [1,n] \uplus B$ is the universe,
\item $P^\mathcal{A} = B$,
\item $Q^\mathcal{A} = \{(a,af) : a\in [1,n]\} \cup (B \times B)$ and
\item the interpretation of the function symbol $h$ is defined as follows, where $g \in B$ is arbitrary:
\begin{displaymath}
h^\mathcal{A}(a,p) = \begin{cases}
ap & \text{if } a \in [1,n] \text{ and }  p \in B,  \\
g & \text{if } a \in B \text{ or } p \in [1,n]. \\
\end{cases}
\end{displaymath}
\end{itemize}
Let $\varphi_k$ be the following sentence:
\begin{displaymath}
\varphi_k = \exists x_1 \cdots \exists x_k \bigwedge_{i=1}^k P(x_i) \wedge 
\forall z \; Q(z, h(\ldots h(h(z,x_1),x_2),\ldots,x_k)).
\end{displaymath}
Then, $(f, B, k)$ is a positive instance if and only if $\mathcal{A} \models \varphi_k$.
\end{proof}

Since $\tmfactor$ is $\mathsf{W}[\textsc{sync}]$-complete \cite{fernaubruchertseifer}, we obtain:

\begin{corollary}
\rm $\mathsf{W}[\textsc{sync}] \subseteq \mathsf{W}^{\func}[2]$.
\end{corollary}

\subsection{Factorization in symmetric groups}

We now show that the various variants of factorization problems for symmetric groups are equivalent with respect to fpt-reductions:
\begin{theorem}\label{theoremsubsetknapsack}
$\sgfactor \equiv^{\fpt} \sgks \equiv^{\fpt} \sgsss$
\end{theorem}

\begin{proof}
By Lemma~\ref{lemma-simple} it suffices to show $\sgsss \lfpt \sgfactor$. 
Let $\pi,\pi_1,\dots,\pi_m \in S_n$ and $k \in \naturals$ be the input data for
$\sgsss$. W.l.o.g. we can assume that $k \geq 4$. For $i < j$ we use $\llbracket i,j\rrbracket $ to denote the cycle $(i,i+1\dots,j)$. Let $n_0=n+k+3$
and $\beta = \llbracket n+1,n_0-1\rrbracket$, which is a cycle of length $k+2$.
For $l \in [1,m+1]$ we define the cycle  $\gamma_l = \llbracket n_0+(l-1)k,n_0+lk\rrbracket$.
Note that consecutive cycles $\gamma_l$ and $\gamma_{l+1}$ intersect in the point $n_0+lk$. 

We define an instance of $\sgfactor$. The group generators are
\begin{alignat}{2}
\tau_{i,j} & \ = \ \pi_j \beta  \prod_{l=i}^j \gamma_l  & \quad \text{ for } &  j \in [1,m], i \in [1,j] \text{ and } \label{pi-ij}  \\
\tau_{i,m+1} & \ =\  \beta  \prod_{l=i}^{m+1} \gamma_l &  \quad \text{ for } &  i \in [1,m+1] . \label{pi-i,m+1}
\end{alignat}
Finally we define the permutation $\tau$ to be factored by
\begin{equation} \label{def-tau}
\tau = \pi \beta^{k+1} \prod_{l=1}^{m+1} \gamma_l .
\end{equation}
For the following, it is important that
\begin{itemize}
\item $\pi, \pi_1, \ldots, \pi_m$ act trivially on all points that are not in $[1,n]$,
\item $\beta$ acts trivially on all points that are not in $[n+1, n_0-1]$, and
\item $\gamma_1, \ldots, \gamma_{m+1}$ act trivially on all points that are not in $[n_0, n_0+(m+1)k]$.
\end{itemize}
Moreover, the three intervals $[1,n]$, $[n+1, n_0-1]$ and $[n_0, n_0+(m+1)k]$ are pairwise disjoint.

We now show that there exist $x_1, \ldots, x_m \in \{0,1\}$ with $\sum_{i=1}^m x_i = k$ and $\pi = \pi_1^{x_1} \cdots \pi_m^{x_m}$ if and only if
$\tau \in \{\tau_{i,j} : 1 \leq i \leq j \leq m+1\}^{\leq k+1}$. We start with the easier direction from left to right.

Suppose that there  exist $x_1, \ldots, x_m \in \{0,1\}$ with $\sum_{i=1}^m x_i = k$ and $\pi = \pi_1^{x_1} \cdots \pi_m^{x_m}$.
In other words, we can write $\pi = \pi_{i_1} \cdots \pi_{i_k}$
with strictly increasing indices $1 \leq i_1 < \dots < i_k \leq m$. 
Note that since the indices are strictly increasing we have $i_d \geq i_{d-1}+1$ for $2 \leq d \leq k$ and $m+1 \geq i_k+1$. Thus,
the permutations  $\tau_{i_{d-1}+1,i_d}$ and $\tau_{i_k+1,m+1}$ are defined. 
We obtain the following (below, the reader finds justifications for the equalities):
\begin{eqnarray*}
 \tau_{1,i_1} \cdot \bigg( \prod_{d=2}^k \tau_{i_{d-1}+1,i_d}\bigg) \cdot \tau_{i_k+1,m+1}  
& \stackrel{\text{(a)}}{=} &  \pi_{i_1} \beta  \prod_{l=1}^{i_1} \gamma_l \cdot \prod_{d=2}^k \bigg( \pi_{i_d} \beta  \!\!\! \prod_{l=i_{d-1}+1}^{i_d} \!\!\! \gamma_l \bigg) \cdot \beta  \!\!\! \prod_{l=i_k+1}^{m+1} \!\! \gamma_l  \\
& \stackrel{\text{(b)}}{=} &  \pi_{i_1} \cdots \pi_{i_k} \beta^{k+1}  \cdot \prod_{l=1}^{i_1} \gamma_l \cdot \prod_{d=2}^k  \prod_{l=i_{d-1}+1}^{i_d} \!\!\!\! \gamma_l 
\cdot   \!\! \prod_{l=i_k+1}^{m+1} \!\!\! \gamma_l  \\
& \stackrel{\text{(c)}}{=} & \pi_{i_1} \cdots \pi_{i_k} \beta^{k+1}  \prod_{l=1}^{m+1} \gamma_l \ \stackrel{\text{(d)}}{=} \ \tau .
\end{eqnarray*}
Equality (a) follows directly from the definition of the $\pi_{i,j}$ in \eqref{pi-ij} and \eqref{pi-i,m+1}.
For equality (b) note that  for all $i,a \in [1,m]$ and $b \in [a+1,m+1]$ the three permutations 
$\pi_i$, $\beta$ and $\prod_{l=a}^b \gamma_l$ pairwise commute since these
permutations act non-trivially on pairwise disjoint domains. Equality (c) is trivial and (d) finally follows from \eqref{def-tau}.
By the above calculation we obtain $\tau \in \{\tau_{i,j} : 1 \leq i \leq j \leq m+1\}^{\leq k+1}$.

Now suppose $\tau \in \{\tau_{i,j} : 1 \leq i \leq j \leq m+1\}^{\leq k+1}$. Then $\tau$ is generated by a sequence of $h \leq k+1$ permutations. Hence there are numbers $j_1,\dots,j_h \in [1,m+1]$ and for 
every $d \in [1,h]$ numbers $i_d \in [1,j_d]$ such that
\begin{equation} \label{eq-tau-iii}
\tau = \prod_{d=1}^h \tau_{i_d,j_d} .
\end{equation}
By projecting onto the domain $[n+1,n_0-1]$ of $\beta$,
this implies
$\beta ^{k+1} = \beta^h$
(note that every $\tau_{i,j}$ contains exactly one copy of $\beta$ and $\tau$ contains $\beta^{k+1}$).
Since $\beta$ is a cycle of length $k+2$ we obtain $h=k+1$. Moreover, projecting  \eqref{eq-tau-iii} onto the union of the domains
of the $\gamma_l$ (which is $[n_0, n_0+(m+1)k]$) yields
\begin{equation}\label{groupfactsubseteq}
\prod_{l=1}^{m+1} \gamma_l  = \prod_{d=1}^{k+1} \prod_{l=i_d}^{j_d} \gamma_l .
\end{equation}
This equation implies two crucial constraints, namely that every cycle $\gamma_l$ ($l \in [1,m+1]$) occurs in the right-hand side of 
\eqref{groupfactsubseteq}  exactly once (Claim~\ref{claim-1}) and the unique occurrence of the cycle $\gamma_r$ 
appears to the left of the unique occurrence of the cycle $\gamma_s$ whenever $r < s$ (Claim~\ref{claim-2}).
From these constraints we can then deduce that the projection of the factorization \eqref{eq-tau-iii}
onto the domain $[1,n]$ of the permutations $\pi, \pi_1, \ldots, \pi_m$ yields an $\sgsss$-solution for the input instance 
$\pi, \pi_1, \ldots, \pi_m$.

\begin{newclaim} \label{claim-1}
If equation~\eqref{groupfactsubseteq} holds then for all $l \in [1,m+1]$ the cycle $\gamma_l$ occurs in the right-hand side exactly once.
\end{newclaim}

\noindent
\emph{Proof of Claim~\ref{claim-1}.}
If a cycle $\gamma_l$ does not appear in the right-hand side of \eqref{groupfactsubseteq}, then the point $n_0 + (l-1)k+1$ is not moved by the right-hand side 
of \eqref{groupfactsubseteq} but it is moved by the left-hand side (recall that $k \geq 4$). Hence, every cycle $\gamma_l$ ($l \in [1,m+1]$) must appear in the right-hand side.

Now suppose that a cycle $\gamma_s$ for some $s \in [1,m+1]$ occurs more than once in the right-hand side. We have
\begin{displaymath}
(n_0+(s-1)k+1) \cdot {\prod_{l=1}^{m+1} \gamma_l} = n_0+(s-1)k+2 .
\end{displaymath}
In the following we consider the first two occurences of $\gamma_s$ in the right-hand side of \eqref{groupfactsubseteq} and suppose that these occurrences are among the first $e \geq 2$ cycles $\gamma_l$ and any further occurrence of $\gamma_s$ comes after the first $e$ cycles. Then we have
\begin{displaymath}
(n_0+(s-1)k+1) \cdot {\prod_{d=1}^{k+1} \prod_{l=i_d}^{j_d} \gamma_l } = (n_0+(s-1)k+3) \cdot {\prod_{d=e+1}^{k+1} \prod_{l=i_d}^{j_d} \gamma_l }.
\end{displaymath}
Here we use the fact that every cycle $\gamma_l$ has length $k+1$ and  $k \geq 4$. 
From \eqref{groupfactsubseteq} we therefore obtain
\begin{equation} \label{eq-proof-claim1}
(n_0+(s-1)k+3) \cdot {\prod_{d=e+1}^{k+1} \prod_{l=i_d}^{j_d} \gamma_l } = n_0+(s-1)k+2.
\end{equation}
Since for every cycle $\gamma_r$ with $r \neq s$ and every $z \in [1,k-1]$ we have
\begin{displaymath}
(n_0+(s-1)k+z) \cdot {\gamma_r} = n_0+(s-1)k+z,
\end{displaymath}
equation \eqref{eq-proof-claim1}
can only hold if the product $\prod_{d=e+1}^{k+1} \prod_{l=i_d}^{j_d} \gamma_l$ contains $k$ further occurrences of $\gamma_s$
(since $\gamma_s$ is a cycle of length $k+1$, we have $(n_0+(s-1)k+3) \cdot {\gamma_s^k} = (n_0+(s-1)k+3) \cdot {\gamma_s^{-1}} = n_0+(s-1)k+2$).
However, the product $\prod_{d=e+1}^{k+1} \prod_{l=i_d}^{j_d} \gamma_l$
contains at most $k+1-e \leq k-1$ further occurrences of $\gamma_s$. By this it follows
\begin{displaymath}
(n_0+(s-1)k+1) \cdot {\prod_{d=1}^{k+1} \prod_{l=i_d}^{j_d} \gamma_l} \neq n_0+(s-1)k+2,
\end{displaymath}
which is a contradiction. This proves Claim~\ref{claim-1}.
\qed

\begin{newclaim} \label{claim-2}
If equation~\eqref{groupfactsubseteq} holds and $r <s$ then $\gamma_r$ appears before $\gamma_s$
in the right-hand side of \eqref{groupfactsubseteq}. 
\end{newclaim}

\noindent
\emph{Proof of Claim~\ref{claim-2}.}
By Claim~\ref{claim-1} we can assume that every cycle $\gamma_l$ is contained at the right-hand side exactly once. 
Assume that $r< s$ and  $\gamma_s$ appears before $\gamma_r$. We will deduce a contradiction.
We can assume that the difference $s-r$ is minimal in the sense that there is no  $y$ such that $r < y < s$  and $\gamma_s$ appears before $\gamma_y$ or $\gamma_y$ appears before $\gamma_r$. In this case we must have $s = r+1$. Otherwise we have $s-r \geq 2$. If we then consider the cycle $\gamma_{r+1}$, we obtain a contradiction to the minimality of $s-r$ in case
$\gamma_s$ appears before $\gamma_{r+1}$ as well as in case 
$\gamma_{r+1}$ appears before $\gamma_s$ (since then also $\gamma_{r+1}$ appears before $\gamma_r$). 

Let $r_1,\dots,r_{m+1} \in \{1,\dots,m+1\}$ with $r_i \neq r_j$ for $i \neq j$ be the sequence of numbers such that
\begin{displaymath}
\prod_{d=1}^{k+1} \prod_{l=i_d}^{j_d} \gamma_l  = \prod_{l=1}^{m+1} \gamma_{r_l}
\end{displaymath}
(this should be seen as an identity between words over the letters $\gamma_l$).
Thus, we have
\begin{equation} \label{eq-gamma-claim-2}
\prod_{l=1}^{m+1} \gamma_{r_l} = \prod_{l=1}^{m+1} \gamma_l .
\end{equation}
Since $s = r+1$,
there are $p<q$ such that $r_p = r+1$ and $r_q = r$.
We obtain the following (justifications for the equalities can be found below):
\begin{equation*} 
(n_0+rk-1) \cdot {\prod_{l=1}^{m+1} \gamma_{r_l} } 
\stackrel{\text{(a)}}{=}  (n_0+rk-1)\cdot {\prod_{l=q}^{m+1} \gamma_{r_l} }
\stackrel{\text{(b)}}{=}  (n_0+rk)\cdot {\prod_{l=q+1}^{m+1} \gamma_{r_l} } 
\stackrel{\text{(c)}}{=}  n_0+rk \ .
\end{equation*}
Equation (a) holds since the point $n_0+rk-1$ is not moved by the cycles $\gamma_i$ with $i \neq r$.
Since $\gamma_r = \gamma_{r_q}$ it follows that $n_0+rk-1$ is not moved by $\prod_{l=1}^{q-1} \gamma_{r_l}$
(recall that every $\gamma_i$ appears exactly once in the product).
Equation (b) holds since $(n_0+rk-1)\cdot \gamma_{r_q} = (n_0+rk-1)\cdot \gamma_r = n_0+rk$.
Finally, (c) holds since $n_0+rk$ is only moved by $\gamma_r = \gamma_{r_q}$ and $\gamma_{r+1} = \gamma_{r_p}$ 
and, since $p<q$, these two cycles do not appear in the product $\prod_{l=q+1}^{m+1} \gamma_{r_l}$.

On the other hand, using similar arguments, we obtain
\begin{alignat*}{2}
(n_0+rk-1)\cdot {\prod_{l=1}^{m+1} \gamma_l } \ = \ & (n_0+rk-1) \cdot {\prod_{l=r}^{m+1} \gamma_l } & \ = \ & (n_0+rk)\cdot {\prod_{l=r+1}^{m+1} \gamma_l }\\
= \ & (n_0+rk+1)\cdot {\prod_{l=r+2}^{m+1} \gamma_l } & \ = \ & n_0+rk+1 \ .
\end{alignat*}
This contradicts \eqref{eq-gamma-claim-2} and shows Claim~\ref{claim-2}.
\qed

\medskip
\noindent
From Claim~\ref{claim-1} it follows that in equation \eqref{groupfactsubseteq} the intervals $[i_d,j_d]$ for $d \in [1,k+1]$ form a partition of $[1,m+1]$.
Moreover, by Claim~\ref{claim-2} we must have $j_d < i_{d+1}$ for all  $d \in [1,k]$. This implies $i_1 = 1$, $i_{d+1} = j_d+1$ for all 
$d \in [1,k]$ and $j_{k+1} = m+1$. Recall that every generator $\tau_{i,j}$ for $1 \leq i \leq j \leq m$ includes the permutation $\pi_j \in S_n$
acting on $[1,n]$.
Hence, by projecting \eqref{eq-tau-iii} onto $[1,n]$ we get $\pi = \pi_{j_1} \pi_{j_2} \cdots \pi_{j_k}$ with $j_1 < j_2 < \cdots < j_k$.
\end{proof}

\subsection{Factorization in the commutative setting}

\subsubsection{Simple lemmas}

In this section, we collect a few simple lemmas that will be needed for the proof of Theorem~\ref{maintheorem}.

Clearly, the $\naturals$-variants of our factorization problems reduce to the corresponding $\integers$-variants. Vice versa, by adding a large
enough element from $\naturals$ (respectively, $\naturals^*$) to all input elements, one gets:

\begin{lemma}  \label{lemma-Z-to-N}
For all $\mathsf{X} \in \{ \mathsf{KS}, \mathsf{SSS},  \mathsf{SSS}^{\neq}\}$ we have
$\mathsf{pX}[\integers] \leq^\fpt \mathsf{pX}[\naturals]$ and $\mathsf{pX}[\integers^*] \leq^\fpt \mathsf{pX}[\naturals^*]$.
\end{lemma}

\begin{proof}
Consider input vectors $a, a_1, \ldots, a_m \in \integers^n$ and let $k \geq 1$ be the parameter. 
All integers are given in binary representation.  We  then compute a vector $b \in \naturals^n$ such that
$a+kb, a_1+b, \ldots, a_m+b \in \naturals^n$ and take the new input elements $a+kb, a_1+b, \ldots, a_m+b$.
\end{proof}

The next lemma is shown by concatenating the binary representations of $d$ integers $c_1, \ldots, c_d$ to a single binary representation with enough zeros inbetween in order
to avoid unwanted carries. 

\begin{lemma} \label{thm-Zn-to-Z}
For all $\mathsf{X} \in \{ \mathsf{KS}, \mathsf{KS}_{\leq},  \mathsf{SSS}, \mathsf{SSS}_{\leq},  \mathsf{SSS}^{\neq}, \mathsf{SSS}^{\neq}_{\leq}\}$ we have
$\mathsf{pX}[\naturals^*] \leq^\fpt \mathsf{pX}[\naturals]$.
\end{lemma}

\begin{proof}
We only show $\Nnsssr  \leq^\fpt \Nsssr$, the reduction also works for the other cases.
Let $A \subseteq \naturals^d$, $a \in \naturals^d$ and $k \geq 1$ be the input data for $\Nnsssr$. 
Every number is given in binary representation.

The idea is then to append the binary representations of $d$ integers $c_1, \ldots, c_d$ into a single binary representation with enough zeros inbetween in order
to avoid unwanted carries. Let $e$ be the maximum entry in a vector from $A$ and let $m = \#(ke)$. Note that $\#(ke)$ denotes the number of bits of $ke$.
Consider a tuple $c = (c_1, \ldots, c_d) \in [0,e]^d$ and let $(b_{i,1} \cdots b_{i,m_i})$ be the binary representation of the number $c_i$ with $b_{i,1}$ the least significant
bit. Note that $b_{i,m_i} = 1$ and $m_i \leq m$. We then encode $c$ by the number $\mathsf{code}(c)$ with the binary representation
\[
(b_{1,1} \cdots b_{1,m_1} \!\!\!\! \underbrace{0 0 \cdots 0}_{\text{$m-m_1$ many}} \!\!\!\! b_{2,1} \cdots b_{2,m_2} \!\!\!\! \underbrace{0 0 \cdots 0}_{\text{$m-m_2$ many}} \!\!\!\!  \cdots
b_{d,1} \cdots b_{d,m_d}) .
\]
Let $\mathsf{code}(A) = \{ \mathsf{code}(b) : b \in A \}$.
Then, there is a subset $S \subseteq A$ of size $k$ such that $\sum_{b \in S} b = a$ if and only if there is a subset $S \subseteq \mathsf{code}(A)$ of size $k$ such that $\sum_{b \in S} b = \mathsf{code}(a)$.
\end{proof}

In \cite{vaidyanathan} it was shown that $\sgfactor$ and $\sgfactorleq$ are equivalent with respect to fpt-reductions.
One direction generalizes easily to the following statement (simply add the neutral element $1$ sufficiently many times to the input list):

\begin{lemma} \label{lemma-simple1} 
For all classes $\mathsf{C}$ from Definition~\ref{def-monoid-classes} and $\mathsf{X} \in \{ \mathsf{F}, \mathsf{KS}, \mathsf{SSS} \}$
 we have $\mathsf{pX_{\leq}[C]} \leq \mathsf{pX[C]}$.
\end{lemma}

\begin{proof}
For $\mathsf{X} \in \{ \mathsf{F}, \mathsf{KS}\}$ it suffices to add the neutral element $1 \in M$ to the input sequence
$a_1, \ldots, a_m$. For $\Csssleq \leq^\fpt \Csss$, we have to add $k$ copies of $1$ to the input sequence.
\end{proof}

\begin{lemma}\label{lemma-simple2}
If $\mathsf{X} \in \{ \mathsf{F}, \mathsf{KS}, \mathsf{SSS}, \mathsf{SSS}^{\neq} \}$ and
$\mathsf{C}$ is one of the classes from Definition~\ref{def-monoid-classes}, then 
$\mathsf{pX[C]}  \leq^\fpt \mathsf{pX_{\leq}[C]}$.
\end{lemma}

\begin{proof}
Let us first explain the idea for $\sgfactor \leq^\fpt \sgfactorleq$ and consider the symmetric group $S_n$ and the elements $\pi, \pi_1, \ldots, \pi_m \in S_n$.
The idea is to work in the group $S_n \times \integers_{k+1} \leq S_{n+k+1}$. For this
take for $\tau$ a cycle of length $k+1$ on the set $[n+1,n+k+1]$ and view $\pi,\pi_1,\dots,\pi_m$ as elements of $S_{n+k+1}$.
Then we have $\pi \in \{ \pi_1, \ldots, \pi_m\}^k$ if and only if $\pi\tau^k \in \{ \pi_1\tau, \ldots, \pi_m\tau\}^{\le k}$. The same idea works (with small modifications) for the classes
$\mathsf{T}_*$, $\mathsf{CycPG}$, $\mathsf{AbPG}$, $\mathsf{FAbG}$, $\integers^*$, and $\naturals^*$.

For $\mathsf{CycPG}$ we have to take a cycle of length $p > \max(n,k)$ for a prime $p$ instead of the
cycle $\tau$ of length $k+1$. 
Since $n$ is given in unary notation, such a prime can be found in polynomial time. Then, for every $\pi \in S_n$, the group $\langle \pi \rangle \times \langle \tau \rangle$
is again cyclic and the cycle $\tau$ has length at least $k+1$.

For $\mathsf{C} = \mathsf{FCycG}$ we explain the idea for $\fcgks \leq^\fpt \fcgksleq$. Let $z,z_1,\dots,z_m \in \integers_n$ and $k \in \naturals$ be the input of $\fcgks$. We will work in the group $\integers_n \times \integers_{nk+1}$. There are numbers $x_1,\dots,x_m \in \naturals$ such that
\begin{displaymath}
\sum_{i=1}^m x_iz_i \equiv z \bmod n \text{ with } \sum_{i=1}^m x_i = k
\end{displaymath}
if and only if
\begin{displaymath}
\sum_{i=1}^m x_i(z_i,1) = (z,k) \text{ with } \sum_{i=1}^m x_i \leq k
\end{displaymath}
in which the computations in the first coordinate are modulo $n$ and in the second modulo $(nk+1)$. Moreover,
since $n$ and $nk+1$ are coprime, we have $\integers_n \times \integers_{nk+1} \cong \integers_{n(nk+1)}$ and the latter is a cyclic group.
An isomorphism $h: \integers_n \times \integers_{nk+1} \to \integers_{n(nk+1)}$ is obtained from Theorem~\ref{CRT}, which also shows
that the values
$h(z_1,1), \ldots, h(z_m,1),h(z,k) \in \integers_{n(nk+1)}$ can be computed in polynomial time. 

Analogously we obtain $\mathsf{pX[FCycG]} \leq^\fpt \mathsf{pX_{\leq}[FCycG]}$ for all $\mathsf{X} \in \{ \mathsf{F}, \mathsf{SSS}, \mathsf{SSS}^{\neq} \}$.
Finally, for $\integers$ we first reduce $\mathsf{pX[\integers]}$ to $\mathsf{pX[\naturals]}$ using Lemma~\ref{lemma-Z-to-N}. The latter is then reduced to 
$\mathsf{pX_{\leq}[\naturals \times \naturals]}$ using the above construction. Finally 
$\mathsf{pX_{\leq}[\naturals \times \naturals]}$ can be reduced to $\mathsf{pX_{\leq}[\naturals]}$
by Lemma~\ref{thm-Zn-to-Z}.
\end{proof}

\begin{lemma} \label{thm-rep-free}
We have $\Csss \leq^\fpt \Csssr$ for each of the classes from Definition~\ref{def-monoid-classes}
.
\end{lemma}

\begin{proof} We first show $\sgsss \leq^\fpt \sgsssr$.
Consider input data $\pi, \pi_1, \ldots, \pi_m \in S_n$ and $k$ for $\sgsss$.
We will work with the group $\integers_2^{2m+1} \times \integers_{2k+2} \times S_n \leq S_{4m+2k+n+4}$.

Consider $i \in [1,m]$. We replace the element $\pi_i$ in the sequence $\pi_1, \ldots, \pi_m$ by the group element
\[ \rho_i = (e^{2m+1}_{2i}, 1, \pi_i) \in \integers_2^{2m+1} \times \integers_{2k+2} \times S_n,
\]
where $e^{2m+1}_{2i} \in \integers_2^{2m+1}$ is the $2m+1$ dimensional unit vector with a $1$ at the $(2i)^{\text{th}}$ coordinate (all other entries are zero).
Note that the $\rho_i$ are pairwise different.
We then add to the sequence all filling elements
\begin{equation} \label{filling}
(0, \ldots, 0, 1, \ldots, 1, 0, \ldots, 0, \mathsf{id}) \in \integers_2^{2m+1} \times \integers_{2k+2} \times S_n ,
\end{equation}
where the $1$'s form a contiguous non-empty block that starts at some position $2i+1$ and ends at some position $2j+1$
with $0 \leq i \leq j \leq m$. 
With $\mathsf{id}$ we denote the identity permutation.
Let $\rho_{m+1}, \ldots, \rho_{m+\ell}$ be a list of all filling elements (the precise order does not matter, since the filling elements pairwise
commute). Note that the elements in the list $\rho_{1}, \ldots, \rho_{m+\ell}$ are pairwise different.

Finally, let $\rho = (1, 1, \ldots, 1, k, \pi) \in \integers_2^{2m+1} \times \integers_{2k+2} \times S_n$. We claim that there is a subset $I \subseteq [1,m]$ such that $|I| = k$ and $\prod_{i \in I} \pi_i = \pi$ 
if and only if there is a subset $J  \subseteq [1,m+\ell]$ with $|J| = 2k+1$ and $\prod_{j \in J} \rho_j = \rho$.\footnote{When writing $\prod_{i \in I} \pi_i$ we assume that the product goes over the set
$I$ in ascending order and similarly for other products.}

First assume that there is $J  \subseteq [1,m+\ell]$ with $|J| = 2k+1$ and $\prod_{j \in J} \rho_j = \rho$. The $k$ in the second last position of $\rho$ forces $J$ to contain exactly
$k$ indices from $[1,m]$. Let $I = J \cap [1,m]$ so that $|I| = k$.
Projecting on the last coordinate yields $\prod_{i \in I} \pi_i = \pi$. 

For the other direction let $I \subseteq [1,m]$ with $|I|=k$ and $\prod_{i \in I} \pi_i = \pi$. 
Consider now the group element $\prod_{i \in I} \rho_i$. Its last entry is $\pi$ and the second last entry is $k$. Moreover, the first 
$2m+1$ entries of $\prod_{i \in I} \rho_i$ contain exactly $k$ ones (all other entries are zero) and two ones are separated by at least one zero.
The zero gaps can be filled with exactly $k+1$ filling elements of the form \eqref{filling}. Hence, there is a subset $J \subseteq [m+1,m+\ell]$
such that $\prod_{i \in I \cup J} \rho_i = \rho$.
 
 The same reduction can be used in a similar way for the other classes from Definition~\ref{def-monoid-classes}.
 For $\mathsf{C} = \mathsf{FCycG}$ assume that we have a $\fcgsss$-instance over the cyclic group $\integers_n$
 with $n$ given in binary notation.
 Let $p_1 < p_2 < \cdots < p_{2m+2}$ be the first $2m+2$ primes that are larger than $2k+1$ and do not divide $n$.
 Since $n$ has at most $\log_2(n)$ different prime divisors we need to check at most the first $2k+1+\lfloor\log_2(n)\rfloor+2m+2$ primes. Hence, we can find these primes in polynomial time.
 Then we can work in the group $\prod_{i=1}^{2m+2} \integers_{p_i} \times \integers_n \cong \integers_{N}$ 
 for $N=n \cdot \prod_{i=1}^{2m+2} p_i$ which is again a cyclic group.
Note that the binary encoding of $N$ can be computed in polynomial time.
 Moreover we can compute the corresponding numbers in $\integers_{N}$ in polynomial time.
 
 For $\mathsf{C} = \mathsf{CycPG}$ we replace the cyclic permutation
 group $G \leq S_n$ by $\prod_{i=1}^{2m+2} \integers_{p_i} \times G$, where $p_1 < p_2 < \cdots < p_{2m+2}$ are primes with $p_1 \geq \max\{n+1, 2k+2\}$.
 This ensures that $\prod_{i=1}^{2m+2} \integers_{p_i} \times G \leq S_{n'}$ is again a cyclic permutation group with $n' = \sum_{i=1}^{2m+2} p_i + n$.
 
 Finally for $\mathsf{C} = \integers$ and  $\mathsf{C} = \naturals$ we obtain with Lemmas~\ref{lemma-Z-to-N} and \ref{thm-Zn-to-Z} the following reduction:
 $\Nsss \lfpt \Zsss \lfpt  \Znsss \lfpt \Znsssr \lfpt \Nnsssr \lfpt \Nsssr \lfpt \Zsssr$.
\end{proof}

\begin{corollary} \label{coro-subset}
Let $\mathsf{C}$ be one of the classes from Definition~\ref{def-monoid-classes}.
Then the problems $\Csss$, $\Csssr$, $\Csssleq$, $\Csssrleq$ 
 are equivalent with respect to fpt-reductions.
\end{corollary}

\begin{proof}
By Lemmas~\ref{lemma-simple1} and \ref{lemma-simple2} we have  
\[ \Csss \equiv^\fpt \Csssleq.\]
By Lemma~\ref{thm-rep-free} we have 
\[ \Csss \equiv^\fpt \Csssr\] ($\Csssr \leq^\fpt \Csss$ is trivial). Moreover, we have
\[ \Csssrleq \leq^\fpt \Csssleq \leq^\fpt \Csss \leq^\fpt  \Csssr \leq^\fpt \Csssrleq, 
\] 
where the final reduction follows from Lemma~\ref{lemma-simple2}.
\end{proof}

For the commutative classes from Definition~\ref{def-monoid-classes} with the exception of $\mathsf{FAbG}$ and $\mathsf{AbPG}$,
we can show the equivalence (with respect to fpt-reductions) for all our factorization problems:

\begin{theorem}\label{maintheorem}
All problems $\mathsf{pX}[\mathsf{C}]$ for 
\begin{itemize}
\item $\mathsf{X} \in \{\mathsf{KS}, \mathsf{KS}_\leq, \mathsf{SSS},  \mathsf{SSS}_\leq, \mathsf{SSS}^{\neq}, \mathsf{SSS}^{\neq}_\leq \}$ and
\item $\mathsf{C} \in \{\mathsf{CycPG}, \mathsf{FCycG},  \naturals, \integers, \naturals^*, \integers^*\}$
\end{itemize}
 are equivalent with respect to fpt-reductions.
 Moreover all these problems are
$\mathsf{W}[1]$-hard and contained in $\mathsf{W}[3]$.
\end{theorem}

We obtain Theorem~\ref{maintheorem} from Theorem~\ref{thm-last-piece} that will be shown in the next section.
The statements concerning $\mathsf{W}[1]$ and $\mathsf{W}[3]$ from Theorem~\ref{maintheorem}
follow from \cite{downeyfellows} (where $\mathsf{W}[1]$-hardness of $\Zsss$ was shown) and 
\cite{bussislam} (where membership of $\Zsssr$ in $\mathsf{W}[3]$ was shown).

\subsubsection{Commutative subset sum problems} \label{sec-commutative}

\begin{lemma} \label{lemma-I-MCK}
For all $\mathsf{X} \in \{\mathsf{KS}, \mathsf{KS}_\leq, \mathsf{SSS},  \mathsf{SSS}_\leq, \mathsf{SSS}^{\neq}, \mathsf{SSS}^{\neq}_\leq \}$  we have
\[ \mathsf{pX}[\mathsf{AbPG}] \leq^\fpt \mathsf{pX}[\mathsf{FAbG}] \quad \text{ and } \quad \mathsf{pX}[\mathsf{CycPG}] \leq^\fpt \mathsf{pX}[\mathsf{FCycG}].
\]
\end{lemma}
This result is an immediate consequence of the following result from  \cite{Iliopoulos88,McKenzieC87}:

\begin{theorem}[c.f.~\cite{Iliopoulos88,McKenzieC87}] \label{lemma-I88}
From a list $\pi_1, \ldots, \pi_m, \pi \in S_n$ of pairwise commuting permutations one can 
compute in polynomial time the following:
\begin{itemize}
\item unary encoded numbers
$n_1, \ldots, n_d$ such that $\langle \pi_1, \ldots, \pi_m \rangle \cong \prod_{j=1}^d \integers_{n_j}$ and
\item an isomorphism
$h : \langle \pi_1, \ldots, \pi_m \rangle \to \prod_{i=1}^d \integers_{n_i}$ that is represented by $h(\pi_1), \ldots, h(\pi_m)$.
\end{itemize}
\end{theorem}
For the proof of Theorem~\ref{maintheorem} we will only need the case of Theorem~\ref{lemma-I88} for cyclic permutation groups.
For this case, we provide a self-contained proof in Appendix~\ref{appendix-cyclic}.

\begin{lemma} \label{fin-cyc-Zn}
$\fcgsss \leq^\fpt \Znsss$
\end{lemma}

\begin{proof}
Let  $\integers_{n}$ be the input group with $n$ given in binary representation and let 
$a, a_1, \ldots, a_m \in [0,n-1]$ be the input elements.

We have to check whether there exists a subset $I \in [1,m]$ such that $|I|=k$ and
\begin{equation} \label{eq-abelian1}
 \sum_{i \in I} a_{i} \equiv a \mod n .
\end{equation}
For every $I \in [1,m]$ of size $k$, \eqref{eq-abelian1} is equivalent to 
\begin{equation} \label{eq-abelian2}
\exists q \in [0,k] : \sum_{i \in I} a_i - q \cdot n = a.
\end{equation}
We define $a' = (k,a)$, $a'_i = (1,a_i)$ and $b = (0,-n)$. These are elements of $\integers^{2}$.
Let us consider the sequence $(a'_1, a'_2 \ldots, a'_m, b,\ldots, b)$
of length $m + k$, where $b$ appears exactly $k$ times. 
Write this sequence as $(c_1, c_2, \ldots, c_{m+k})$.
Then there is $I \in [1,m]$ such that $|I|=k$ and \eqref{eq-abelian2} holds, if and only if 
there is a subset $J \subseteq [1,m+k]$ such that $|J| \leq 2k$ and
\begin{equation} \label{eq-abelian3}
\sum_{j \in J} c_j = a' .
\end{equation}
Finally, by adding the zero vector to the sequence $(c_1, c_2, \ldots, c_{m+k})$, we can replace the condition $|J| \leq 2k$ by the condition $|J| = 2k$.
\end{proof}

\begin{lemma} \label{N-CycPG}
$\Nsss \lfpt \cpgsss$
\end{lemma}

\begin{proof}
The result is an easy consequence of the Chinese remainder theorem (Theorem~\ref{CRT}). To see this, 
let $a, a_1,\dots,a_m \in \naturals$ and $k \in \naturals$ be the input data for $\Nsss$.
Let $e = \max\{a_1,\dots,a_m\}$  and $d$ be the smallest number such that 
\begin{displaymath}
\prod_{i=1}^d p_i > ke,
\end{displaymath}
where $p_1 < \dots < p_d$ are the first $d$ primes. Let $N = \prod_{i \in [1, d]} p_i$ so that $ke < N$.
We have $d \in \mathcal{O}(\log k + \log e)$ and $p_d \in \mathcal{O}(d \log d)$.
The list $p_1, \ldots, p_d$ with every $p_i$ in \emph{unary} encoding can be computed in time $\poly(\log k, \log e)$.

Note that our $\Nsss$ instance has no solution if $a > ke$. In this case the reduction returns some negative $\cpgsss$-instance.
Hence we assume that $a \leq ke < N$ in the following. 

The group we are working with is $G = \prod_{i=1}^{d} \integers_{p_i}$.
Since the primes are chosen pairwise different, Theorem~\ref{CRT} implies that the group $G$ is
isomorphic to the cyclic group $\integers_N$.
The group $G$ can be embedded into  $S_n$ for $n=\sum_{i \in [1, d]} p_i$ by identifying the direct factors $\integers_{p_i}$ of $G$ with disjoint cycles of length $p_i$
($i \in [1,d]$).  For $i \in [1,m]$ and $j \in [1,d]$ let
\[
z_{i,j} = a_i \bmod p_j \in [0, p_j-1] \ \text{ and } \ z_j = a \bmod p_j \in [0, p_j-1].
\]
Then, for every subset $I \subseteq [1,m]$ of size $k$ we have $\sum_{i \in I} a_i = a$ if and only if
$\sum_{i \in I} a_i \equiv a \mod N$ if and only if
$\sum_{i \in I} z_{i,j} \equiv z_j \mod p_j$ for all $j \in [1,d]$. Here we use the Chinese remainder theorem 
for the second equivalence.
\end{proof}

\begin{corollary}  \label{coro-comm-subset}
All problems $\mathsf{pX}[\mathsf{C}]$ for 
\begin{itemize}
\item $\mathsf{X} \in \{\mathsf{SSS},  \mathsf{SSS}_\leq, \mathsf{SSS}^{\neq}, \mathsf{SSS}^{\neq}_\leq \}$ and
\item $\mathsf{C} \in \{\mathsf{CycPG}, \mathsf{FCycG},  \naturals, \integers, \naturals^*, \integers^*\}$
\end{itemize}
 are equivalent with respect to fpt-reductions.
\end{corollary}

\begin{proof}
We have
\begin{alignat}{2}
\cpgsss \quad & \leq^\fpt  \quad \fcgsss  & & \qquad \text{(Lemma~\ref{lemma-I-MCK})} \nonumber \\
& \leq^\fpt  \quad \Znsss  & & \qquad \text{(Lemma~\ref{fin-cyc-Zn})} \nonumber \\
& \leq^\fpt  \quad \Nnsss & & \qquad \text{(Lemma~\ref{lemma-Z-to-N})} \nonumber \\
& \leq^\fpt  \quad \Nsss  & & \qquad \text{(Lemma~\ref{thm-Zn-to-Z})} \nonumber \\
& \leq^\fpt  \quad \cpgsss . & & \qquad \text{(Lemma~\ref{N-CycPG})} \nonumber
\end{alignat}
Finally from Corollary~\ref{coro-subset} we obtain for every
$\mathsf{C} \in \{\mathsf{CycPG}, \mathsf{FCycG},  \naturals, \integers, \naturals^*, \integers^*\}$
the equivalence of the problems $\Csss, \Csssr, \Csssleq, \Csssrleq$ with respect to fpt-reductions.
\end{proof}

The missing piece in the proof of Theorem~\ref{maintheorem} is the following:

\begin{theorem} \label{thm-last-piece}
If $\mathsf{C}$ is one of the commutative classes from Definition~\ref{def-monoid-classes}
then \[ \Csss \leq^\fpt \Cksleq \leq^\fpt  \Cks \lfpt \Csss. \]
\end{theorem}

\begin{proof}
By Lemmas~\ref{lemma-simple} and \ref{lemma-simple1} the statement $\Cksleq \leq^\fpt  \Cks \lfpt \Csss$ holds for every class $\mathsf{C}$.
It remains to show $\Csss \leq^\fpt \Cksleq$.
The idea is similar to the proof of Lemma~\ref{thm-rep-free}.
It suffices to show $\Csssr \leq^\fpt \Cksleq$.
We first prove the theorem for $\mathsf{C} = \integers^*$. 

Let $A \subseteq \integers^n$ be finite, $a \in \integers^n$ and $k$ be the parameter.
Let $A = \{ a_1, \ldots, a_m\}$. We then move to the group $\integers^{m+n+1}$. For $i \in [1,m]$ let
$e_i \in \integers^m$ be the $i^{\text{th}}$ unit vector, i.e., the vector with $m$ entries, where the $i^{\text{th}}$ entry
is $1$ and all other entries are zero.
We define $b_i = (1,e_i,a_i) \in \integers^{m+n+1}$ and $b = (k,1^m,a) \in \integers^{m+n+1}$.
We then add all filling vectors $(0, 0^i, 1^j, 0^l, 0^n)$, where $i+j+l = m$. Consider the set
$B = \{ b_1, \ldots, b_m\} \cup F$, where $F$ is the set of all filling vectors.

We claim that there is a subset $S \subseteq A$ of size $k$ with $\sum_{x \in S} x = a$ if and only if
$b \in B^{\leq 2k+1}$.\footnote{Note that when writing $b \in B^{\leq 2k+1}$ we use the multiplicative 
notation for the abelian group $\integers^{m+n+1}$.}
 First assume that $\sum_{x \in S} x = a$ with $|S|=k$. Let $S = \{ a_{i_1}, \ldots, a_{i_k} \}$
with $i_1 < i_2 < \cdots < i_k$. Then $\sum_{j=1}^k b_{i_j} = (k, v, a)$, where $v \in \{0,1\}^m$ has exactly $k$ ones.
Hence, by adding at most $k+1$ filling vectors to $\sum_{j=1}^k b_{i_j}$, we obtain $b = (k,1^m,a)$.

For the other direction, assume that $b \in B^{\leq 2k+1}$ and let $b = \sum_{i=1}^l c_i$ with $c_i \in B$ and $l \leq 2k+1$.
Inspecting the first coordinate shows that the sum $\sum_{i=1}^l c_i$ contains exactly $k$ vectors from 
$\{ b_1, \ldots, b_m\}$. Moreover, since the entries at positions $2, \ldots, m+1$ in $b$ are $1$, these $k$ vectors
must be pairwise different. By projecting onto the last coordinate we obtain a subset $S \subseteq A$ of size $k$ with $\sum_{x \in S} x = a$.

The above idea also works for $\mathsf{C} \in \{\mathsf{CycPG}, \mathsf{AbPG}, \mathsf{FAbG}, \integers \}$.
For $\mathsf{C} = \integers$ we can use the above reduction together with $\Znksleq \leq^\fpt \Zksleq$, which is obtained as follows:
\begin{alignat*}{2}
\Znksleq \quad & \leq^\fpt  \quad \Znks  & & \qquad \text{(Lemma~\ref{lemma-simple1})} \\
& \leq^\fpt  \quad \Nnks  & & \qquad \text{(Lemma~\ref{lemma-Z-to-N})} \\
& \leq^\fpt  \quad \Nks & & \qquad \text{(Lemma~\ref{thm-Zn-to-Z})} \\
& \leq^\fpt  \quad \Nksleq . & & \qquad \text{(Lemma~\ref{lemma-simple2})} 
\end{alignat*}
Finally, for $\mathsf{C} = \mathsf{FCycG}$ we obtain 
\[ \fcgsss \equiv^{\fpt} \cpgsss \lfpt \cpgksleq \lfpt \fcgksleq\] with 
Corollary~\ref{coro-comm-subset} and Lemma~\ref{lemma-I-MCK}.
\end{proof}

\subsubsection{Abelian groups}

Clearly, the problems $\apgks$ and $\apgsss$  are $\mathsf{W}[1]$-hard; this follows from the corresponding results for $\cpgfactor$ and $\cpgsss$.
The best upper bound that we can show is $\mathsf{W}[5]$:

\begin{theorem} \label{thm-W5}
The problems $\afgsss$ and $\afgks$ belong to $\mathsf{W}[5]$.
\end{theorem}
\begin{proof}
It suffices to show $\afgsss \in \mathsf{W}[5]$ because by Lemma~\ref{lemma-simple} we have $\afgks \lfpt \afgsss$. Let 
$t_i=(t_{i,1},\dots,t_{i,d}) \in \prod_{j=1}^d \integers_{n_j}$ for $i \in [1,m]$, 
$u=(u_1,\dots,u_d) \in \prod_{j=1}^d \integers_{n_j}$ and $k \in \naturals$ be the input for $\afgsss$. 
We can assume that $t_{i,j}, u_j \in [0,n_j-1]$ for all $i \in [1,m]$, $j \in [1,d]$.

First we give a reduction to an equivalent problem, where we avoid computations modulo $n_j$. We have
\begin{displaymath}
\exists x_1,  \ldots, x_m \in \{0,1\} : \sum_{i=1}^m x_i = k \text{ and }  \forall j \in [1,d] : \sum_{i=1}^m x_i t_{i,j} \equiv u_j \bmod n_j 
\end{displaymath}
if and only if 
\begin{displaymath}
\exists x_1,  \ldots, x_m \in \{0,1\} : \sum_{i=1}^m x_i = k \text{ and }  \forall j \in [1,d]  \; \exists y \in [0,k-1] : \sum_{i=1}^m x_i t_{i,j} = u_j + y \cdot n_j .
\end{displaymath}
The values $u_{y,j} := u_j + y \cdot n_j \in \naturals$ for $j \in [1,d]$ and $y \in [0,k-1]$
can be precomputed in $\mathsf{FPT}$. 
Now the input for our new problem consists of the tuples $t_1,\dots,t_m$ as defined above, 
the parameter $k \in \naturals$, and a matrix
\begin{displaymath}
 \begin{bmatrix}
u_{0,1} & \dots & u_{0,d}\\
\vdots & \ddots & \vdots\\
u_{k-1,1} & \dots & u_{k-1,d-1}
\end{bmatrix} \in \naturals^{k \times d} .
\end{displaymath}
This representation allows us to do all computations in $\naturals$.

Now we present  an algorithm for the above problem. Most parts of the algorithm are identical to the $\mathsf{W}[3]$-algorithm of \cite{bussislam} for $\Zsssr$. 
It is straightforward to transform this algorithm into a $\Sigma_{5,2}$-formula that only depends on $k$.
As in \cite{bussislam} we also represent all numbers in base $r$ with $r = 2^{\lceil \log k \rceil} \in [k,2k]$ by taking blocks of $\lceil \log k \rceil$ bits in the binary representations of the numbers.
Let us write 
\[
(b_0, b_1, \ldots, b_p)_r = \sum_{i=0}^p b_i r^i 
\]
for $b_0, \ldots, b_p \in [0,r-1]$.
Let $t_{i,j,p}$ (respectively, $u_{y,j,p}$) be the $p^{\text{th}}$ base-$r$ digit of the number $t_{i,j}$ (respectively, $u_{y,j}$). We start with $p=0$.
Assume that $m_j+1$ is the maximal number of digits in the 
base-$r$ expansions of the
numbers $t_{i,j}$, $u_{y,j}$ for $i \in [1,m]$, $y \in [0,k-1]$. The algorithm consists of the following steps:
\begin{enumerate}[(a)]
\item Existentially guess $k$ tuples $t_{i_1}, \ldots, t_{i_k}$ ($i_1 < i_2 < \cdots < i_k$).
\item Universally guess a coordinate $j \in [1,d]$.
\item Existentially guess a  $y \in [0,k-1]$. The remaining goal is to check whether 
$u_{y,j} = t_{i_1,j} + \cdots + t_{i_k,j}$. 
We define the following base-$r$ digits, where $p \in [0,m_j]$:
\begin{eqnarray*}
s_p  &=& (t_{i_1,j,p} + \cdots + t_{i_k,j,p}) \bmod r \in [0,r-1], \\
s_{m_j+1} & = & 0, \\
c_0 & = & 0, \\
c_{p+1} &=& (t_{i_1,j,p} + \cdots + t_{i_k,j,p}) \;\mathrm{div}\; r \in [0,r-1] .
\end{eqnarray*}
Note that $t_{i_1,j,p} + \cdots + t_{i_k,j,p} \leq k(r-1) \leq r(r-1) = r^2-r$ and
\[
t_{i_1,j} + \cdots + t_{i_k,j} = \sum_{p = 0}^{m_j+1} (s_p r^p +  c_p r^{p}) = \underbrace{(s_0, \ldots, s_{m_j+1})_r}_{v} + \underbrace{(c_0, \ldots, c_{m_j+1})_r}_{w} .
\]
The new goal is to check whether  $v+w = u_{y,j}$.
Note that when adding $v$ and $w$, the maximal carry that can arrive at a certain position is $1$ (since $2 (r-1) + 1 = (r-1,1)_r$).
\item Universally select a position $p \in [0,m_j+1]$.
\item Compute $\delta = (u_{y,j,p} - s_p - c_p) \bmod r$. Reject if $\delta$ is neither $0$ nor $1$.
\item If $\delta=0$ then do the following (we verify that no carry arrives at position $p$):
\begin{enumerate}[(i)]
\item Universally select a position $p' \in [0,p-1]$ that generates a carry (i.e., $s_{p'} + c_{p'} \geq r$).
\item Existentially select a position $q \in [p'+1,p-1]$ such that a carry from position $q-1$ is not propagated through position $q$
(i.e., $s_q + c_q < r-1$).
\end{enumerate}
\item If $\delta=1$ then do the following (we verify that a carry arrives at position $p$):
\begin{enumerate}[(ii)]
\item Universally select a position $p' \in [-1, p-1]$.
\item If $p' = -1$ or the digits at position $p'$ do not propagate a carry 
(i.e., $s_{p'} + c_{p'} < r-1$) then existentially select a position $q \in [p'+1,p-1]$ such that a carry is generated at position $q$ (i.e., $s_{q} + c_{q} \geq r$).
\end{enumerate}
\end{enumerate}
It is straightforward (but tedious) to translate the above algorithm into a $\Sigma_{5,2}$-sentence that only depends on the parameter $k$.
The existential (respectively, universal) guesses correspond to existential (respectively, universal) quantifiers in the formula.
The constructed formula is model-checked in a structure that is constructed from the input numbers.
For each of the following elements we add an element to the universe of the structure:
\begin{enumerate}[(1)]
\item every $i \in [1,m]$, where $i$ represents the tuple $t_i$,
\item every dimension $j \in [1,d]$,
\item every $y \in [0,k-1]$,
\item every position $p$ that exists in one of the base-$r$ expansions of the numbers $t_{i,j}$ and $u_{y,j}$, and
\item all $\alpha \in [0,r-1]$ that stand for the base-$r$ digits.
\end{enumerate}
Moreover, we need the following predicates:
\begin{itemize}
\item four unary predicates $I$, $J$, $Y$, and $P$ that hold exactly for the elements from (1), (2), (3), and (4), respectively,
\item a 4-ary relation $\mathsf{dig}_t$ that contains the tuple $(i, j, p, \alpha)$ if $\alpha \in [0,r-1]$ is the digit at position $p$ in the number $t_{i,j}$,
\item  a 4-ary relation $\mathsf{dig}_u$ that contains the tuple $(y, j, p, \alpha)$ if $\alpha \in [0,r-1]$ is the digit at position $p$ in the number $u_{y,j}$,
\item the natural linear order $<_P$ on the positions from (4), and
\item the natural linear order $<_r$  on $[0,r-1]$.
\end{itemize}
Note that the sets from (1)--(5) do not have to be disjoint. 

The $\Sigma_{5,2}$-formula starts with the prefix 
\begin{equation} \label{quant-prefix}
\exists i_1 \exists i_2 \cdots \exists i_k : \bigwedge_{l=1}^k I(i_l) \wedge \forall j : J(j) \to \exists y : Y(y) \wedge \forall p : P(p) \to \cdots .
\end{equation}
These quantifiers correspond to the guesses in steps (a)--(d) of the above algorithm.
For the remaining part of the formula 
we need for every $a \in [0,r-1]$ formulas 
$R_{a}(i_1, \ldots, i_k, j ,p)$ and $Q_{a}(i_1, \ldots, i_k, j ,p)$, which express
 $(t_{i_1,j,p} + \cdots + t_{i_k,j,p}) \bmod r = a$ and  
 $(t_{i_1,j,p} + \cdots + t_{i_k,j,p}) \; \mathrm{div} \; r = a$, respectively.  
 Here, we use  the variables quantified in \eqref{quant-prefix}.
 Define the finite set 
 \[ A_a = \{ (a_1, \ldots, a_k) \in [0,r-1]^k : (a_1 + \cdots + a_k) \bmod r = a \} \subseteq [0,r-1]^k .
 \]
 It can be precomputed from $k$ and its size only depends on $k$.
 Then we define the quantifier-free formula $R_{a}(i_1, \ldots, i_k, j ,p)$ by
 \[
 \bigvee_{ (a_1, \ldots, a_k) \in A_a} \bigwedge_{1 \le l \le k} \mathsf{dig}_t(i_l, j, p,a_l).
 \]
The formula $Q_{a}(i_1, \ldots, i_k, j ,p)$ can be obtained analogously. Note that these formulas only
depend on the parameter $k$.
With the formulas $R_{a}(i_1, \ldots, i_k, j ,p)$ and $Q_{a}(i_1, \ldots, i_k, j ,p)$ we can
access the numbers $s_p, c_p \in [0,r-1]$ from step (c) of the algorithm. Using them, one can 
construct analogously to $S_{a}(i_1, \ldots, i_k, j ,p)$
quantifier-free formulas expressing 
 conditions like $(u_{y,j,p} - s_p - c_p) \bmod r = 0$ and $(u_{y,j,p} - s_p - c_p) \bmod r = 1$
 (see steps (e)--(g) in the algorithm). 
We leave the further details to the reader.
\end{proof}

\section{Change making problems} \label{sec-change-making}

Note that in the following proofs for a positive integer $N \in \naturals$ we use $\#(N)$ to denote the number of bits of $N$.

The problem $\Nksleq$ (respectively, $\Nsssleqr$) was studied in \cite{goebbels} under the name {\sf p-unbounded-change-making}
(respectively,  {\sf p-0-1-change-making}). A third variant from \cite{goebbels} is:

\begin{problem}[\textsf{p-bounded-change-making} \cite{goebbels}]\label{p-bcm}~\\
Input: binary encoded numbers $c, c_1, b_1 \ldots, c_m, b_m \in \naturals$ with $c_i \neq c_j$ for $i \neq j$ and $k \in \naturals$ (the parameter)\\
Question: Are there $x_i \in [0, b_i]$ $(i \in [1,m])$ with $\sum_{i=1}^m x_i c_i = c$ and $\sum_{i=1}^m x_i \leq k$?
\end{problem}
It was shown in \cite{goebbels} that these change-making problems are $\mathsf{W}[1]$-hard and contained in $\mathsf{XP}$. 
For {\sf p-unbounded-change-making} and {\sf p-0-1-change-making} we obtain membership in $\mathsf{W}[3]$ from 
Theorem~\ref{maintheorem}. Also \textsf{p-bounded-change-making} belongs to $\mathsf{W}[3]$. For this, notice that
 $\Nksleq \lfpt \textsf{p-bounded-change-making}$: if we set $b_i = k$ for all $i \in [1,m]$ in \textsf{p-bounded-change-making},
 we obtain $\Nksleq$. Finally, also $\textsf{p-bounded-change-making} \lfpt \Nsssleqr$ holds: 
 because of the restriction $\sum_{i=1}^m x_i \leq k$ in \textsf{p-bounded-change-making}, 
 we can assume that $b_i \leq k$ for all $i \in [1,m]$. We obtain an instance of $\Nsssleq$ by duplicating every $c_i$ in a
 \textsf{p-bounded-change-making}-instance $k$ times.

\begin{theorem} 
The problem {\sf p-bounded-change-making} is equivalent with respect to fpt-reductions to the problems from Theorem~\ref{maintheorem} and therefore belongs to 
$\mathsf{W}[3]$.
\end{theorem}
In  \cite{goebbels}, approximate versions of the change-making problems, where a linear objective function is minimized, were defined:
 
\begin{problem}[\textsf{p-unbounded-change-approx} \cite{goebbels}]\label{p-uca}~\\
Input: binary encoded numbers $c, c_1,\dots,c_m\in \naturals$ with $c_i \neq c_j$ for $i \neq j$, a linear function 
$f(x,y)=ax+by$ given by $a,b \in \naturals$ and $k \in \naturals$ (the parameter)\\
Question: Are there $x_1,\dots,x_m \in \naturals$ with $f(\sum_{i=1}^m x_ic_i-c,\sum_{i=1}^m x_i) \leq k$ and $\sum_{i=1}^m x_ic_i \geq c$?
\end{problem}

\begin{problem}[\textsf{p-bounded-change-approx} \cite{goebbels}]\label{p-bca}~\\
Input: binary encoded numbers  $c,c_1,b_1\dots,c_m,b_m\in \naturals$ with $c_i \neq c_j$ for $i \neq j$, a linear function $f(x,y)=ax+by$  given by $a,b \in \naturals$ and $k \in \naturals$
(the parameter)\\
Question: Are there  $x_1, \ldots, x_m \in [0, b_i]$  with $f(\sum_{i=1}^m x_ic_i-c,\sum_{i=1}^m x_i) \leq k$ and
 $\sum_{i=1}^m x_ic_i \geq c$?
\end{problem}
Finally, one obtains the problem \textsf{p-0-1-change-approx} by fixing $b_i$ to $1$ for all $i \in [1,m]$ in \textsf{p-bounded-change-approx} \cite{goebbels}.
Also for the above approximation variants, $\mathsf{W}[1]$-hardness and containment in \textsf{XP} was shown in
 \cite{goebbels}. With our results we can improve the upper bounds:

\begin{theorem} \label{thm-W3-approx}
The problems $\Nsss$, {\sf p-unbounded-change-approx}, {\sf p-bounded-change-approx} and {\sf p-0-1-change-approx}  are equivalent under fpt-reductions 
when restricted to functions
$f(x,y) = ax+by$
with $(a,b) \in \naturals \times \naturals \setminus \{(a,0) : a \geq 1\}$. Therefore the problems belong to $\mathsf{W}[3]$.
\end{theorem}

\begin{proof}
We prove the theorem in the following order:
\begin{enumerate}[(i)]
\item $\Nksleq \lfpt$ \textsf{p-unbounded-change-approx},
\item \textsf{p-unbounded-change-approx} $\lfpt$ \textsf{p-bounded-change-approx},
\item $\Nsssleqr$ $\lfpt$ \textsf{p-0-1-change-approx},
\item \textsf{p-0-1-change-approx} $\lfpt$ \textsf{p-bounded-change-approx},
\item \textsf{p-bounded-change-approx} $\lfpt$ $\Zsssleq$.
\end{enumerate}
Since (iv) is trivial and (iii) can be proven analogously to (i) we only show the reductions for (i), (ii) and (v).
For (i) notice that an instance of $\Nksleq$ can be transformed into an equivalent instance
of \textsf{p-unbounded-change-approx} by taking the linear function $f(x,y) = (k+1)x + y$.

For (ii) let $c, c_1,\ldots,c_m, f(x,y) = ax+by$, and $k$ be the input data for
 \textsf{p-unbounded-change-approx}. By setting $b_i=k$ for $i \in [1,m]$ we obtain an equivalent instance of \textsf{p-bounded-change-approx} if $b > 0$. If $b=0$ then also $a=0$ and hence the condition $f(\sum_{i=1}^m x_ic_i-c,\sum_{i=1}^m x_i) \leq k$ is always satisfied.
The \textsf{p-unbounded-change-approx}-instance is then positive if and only if $c=0$ or $c_i > 0$ for at least one $i \in [1,m]$.

It remains to show (v). Let $c, c_1,b_1,\dots,c_m,b_m, f(x,y) = ax+by$, and $k$ be the input data for \textsf{p-bounded-change-approx}.
If $a=0=b$ the instance is positive if and only if $\sum_{i=1}^m b_ic_i \geq c$. 

Suppose now that $a = 0$ and $b \geq 1$ and w.l.o.g.~$c_1 > c_2 > \cdots > c_m$.
We show that it can be checked in polynomial time whether the \textsf{p-bounded-change-approx}-instance is positive.
The latter holds if and only if there exist 
$x_i \in [0, b_i]$ $(i \in [1,m])$ with
 $\sum_{i=1}^m x_ic_i \geq c$ and $\sum_{i=1}^m x_i  \leq \lfloor k/b\rfloor$. We can w.l.o.g.~assume that $b_i > 0$
 for all $i \in [1,m]$. If $\sum_{i=1}^m b_i \leq \lfloor \frac{k}{b} \rfloor$ then it suffices to check whether
$\sum_{i=1}^m b_ic_i \geq c$.
If $\sum_{i=1}^m b_i > \lfloor \frac{k}{b} \rfloor$ then 
we compute the unique $l \in [0,m-1]$ and $b'_{l+1}< b_{l+1}$ such that
$\sum_{i=1}^l b_{i} + b'_{l+1} = \left\lfloor \frac{k}{b} \right\rfloor$.
It then suffices to check whether
$\sum_{i=1}^l b_{i}c_{i} + b'_{l+1}c_{l+1} \geq c$.

The case that remains is $a,b \geq 1$.  We define the following integers:
\begin{alignat*}{2}
d &= c \cdot 2^{\#(ka)+\#(kb)} \qquad & d_i &= b-1 + c_i \cdot 2^{\#(ka)+\#(kb)}  \text{ for } i \in [1,m] \\
d_0 &= -1 & d_{m+1} &= (1-a)  2^{\#(kb)}  - 2^{\#(ka)+\#(kb)} \\
& & d_{m+2} &= 2^{\#(kb)} .
\end{alignat*}
Moreover, we set $k = b_0 = b_{m+1} = b_{m+2}$.

\begin{newclaim} \label{claim1appendix}
The \textsf{p-bounded-change-approx}-instance $c, c_1,b_1,\dots,c_m,b_m, f(x,y) = ax+by$, $k$ is positive if and only if
there exist $x_i \in [0, b_i]$ $(0 \leq i \leq m+2)$ such that
\begin{equation*}
\sum_{i=0}^{m+2} x_id_i = d \quad \text{ and } \quad
\sum_{i=0}^{m+2} x_i \leq k .
\end{equation*}
\end{newclaim}

\noindent
\emph{Proof of Claim~\ref{claim1appendix}.}
First suppose the \textsf{p-bounded-change-approx}-instance is positive. Then there exist $x_i \in [0, b_i]$ with 
\begin{equation*}
\sum_{i=1}^m x_ic_i  \geq  c \quad\text{ and }\quad
a\bigg(\sum_{i=1}^m x_ic_i - c\bigg) + b\sum_{i=1}^m x_i  \leq  k.
\end{equation*}
By choosing $x_0 = (b-1)\sum_{i=1}^m x_i$ and $x_{m+1} = \sum_{i=1}^m x_ic_i - c$ and $x_{m+2} = (a-1)(\sum_{i=1}^m x_ic_i - c)$ we obtain
\begin{align*}
\sum_{i=0}^{m+2} x_id_i=\phantom{ }&(b-1)\big(\sum_{i=1}^m x_i\big)d_0 + \sum_{i=1}^m x_id_i + \big(\sum_{i=1}^m x_ic_i - c\big)d_{m+1} + (a-1) \big(\sum_{i=1}^m x_ic_i - c\big)d_{m+2}\\
=\phantom{ }&-(b-1)\sum_{i=1}^m x_i + \sum_{i=1}^m x_i(b-1) + \sum_{i=1}^m x_ic_i \cdot 2^{\#(ka)+\#(kb)}\\
\phantom{ }&+ \big(\sum_{i=1}^m x_ic_i - c\big)(1-a) \cdot 2^{\#(kb)} - \big(\sum_{i=1}^m x_ic_i - c\big) \cdot 2^{\#(ka)+\#(kb)}\\
\phantom{ }&+ (a-1)\big(\sum_{i=1}^m x_ic_i - c\big) \cdot 2^{\#(kb)}\\
=\phantom{ }&\sum_{i=1}^m x_ic_i \cdot 2^{\#(ka)+\#(kb)} - \big(\sum_{i=1}^m x_ic_i - c\big) \cdot 2^{\#(ka)+\#(kb)}\\
=\phantom{ }&c \cdot 2^{\#(ka)+\#(kb)} = d
\end{align*}
and
\begin{align*}
\sum_{i=0}^{m+2} x_i&=(b-1)\sum_{i=1}^m x_i + \sum_{i=1}^m x_i + \sum_{i=1}^m x_ic_i - c + (a-1)\big(\sum_{i=1}^m x_ic_i - c\big)\\
&=a\big(\sum_{i=1}^m x_ic_i - c\big) + b\sum_{i=1}^m x_i \leq k .
\end{align*}
Therefore we have $x_0,x_{m+1},x_{m+2} \in [0,k]$. Thus $x_i \in [0,b_i]$ for all $i \in [0,m+2]$.

For the other direction in Claim~\ref{claim1appendix} assume that $x_i \in [0,b_i]$ are such that 
\begin{equation}\label{approxchangemakingeq1-app}
\sum_{i=0}^{m+2} x_id_i = d \quad \text{ and } \quad \sum_{i=0}^{m+2} x_i \leq k .
\end{equation}
We bring the left equation in~\eqref{approxchangemakingeq1-app} in a form where all numbers are positive:
\begin{displaymath}
\sum_{i=1}^m x_id_i + x_{m+2}d_{m+2} = d + x_0(-d_0) + x_{m+1}(-d_{m+1}), 
\end{displaymath}
or, after plugging in the values of the $d_i$ and rearranging,
\begin{alignat*}{5}
    & & \sum_{i=1}^m x_i(b-1) & \ + \  & x_{m+2}  & \ 2^{\#(kb)}  & & \ + \ \bigg(\sum_{i=1}^m x_i c_i \bigg) & & \ 2^{\#(ka)+\#(kb)}   \\
= &  & x_0 &  \ + \  & x_{m+1} (a-1) & \ 2^{\#(kb)}  & & \ + \ (c+x_{m+1})  & & \  2^{\#(ka)+\#(kb)}   . 
\end{alignat*}
Note that $\sum_{i=1}^m x_i(b-1)$ and $x_0$ are bounded by $kb$ and therefore have at most $\#(kb)$ bits. 
All other terms in the above equation are multiplied with $2^{\#(kb)}$ and therefore do not affect the first $\#(kb)$ bits. 
Similarly, $x_{m+2}$ and $x_{m+1} (a-1)$ are bounded by $\#(ka)$ and therefore have at most $\#(ka)$ bits.
We therefore split the above equation into three simultaneous equations:
\begin{eqnarray}
\sum_{i=1}^m x_i(b-1) &=& x_0 \label{approxchangemakingeq2-app} \\
x_{m+2}  &=& x_{m+1} (a-1) \label{approxchangemakingeq3-app} \\
\sum_{i=1}^m x_ic_i &=& c  + x_{m+1} . \label{approxchangemakingeq4-app}
\end{eqnarray}
By~\eqref{approxchangemakingeq4-app} we have $x_{m+1} = \sum_{i=1}^m x_ic_i - c$. Together with~\eqref{approxchangemakingeq3-app} this gives us $x_{m+2}=(a-1)(\sum_{i=1}^m x_ic_i - c)$. Adding up all variables and using~\eqref{approxchangemakingeq2-app} yields
\begin{eqnarray*}
k \geq \sum_{i=0}^{m+2} x_i & = & \sum_{i=1}^m x_i(b-1) + \sum_{i=1}^m x_i + \sum_{i=1}^m x_ic_i - c + (a-1)\bigg(\sum_{i=1}^m x_ic_i - c\bigg)\\
& = & a\bigg(\sum_{i=1}^m x_ic_i - c\bigg) + b\sum_{i=1}^m x_i .
\end{eqnarray*}
Moreover since we have $x_{m+1} \geq 0$, \eqref{approxchangemakingeq4-app} finally implies $\sum_{i=1}^m x_ic_i \geq c$ which proves the claim.
\qed

\medskip
\noindent
By Claim~\ref{claim1appendix} we obtain an equivalent instance of $\Zsssleq$ by defining a list where for $i \in [0,m+2]$ the number $d_i$ is contained exactly $b_i$ times.
\end{proof}

\subsection{NP-hardness of the approximative change-making problems for \texorpdfstring{$a > 0$}{a>0} and \texorpdfstring{$b=0$}{b=0}}

\begin{theorem}\label{np-completeslices}
Let $a \geq 1$ and $b=0$. Then for all $d \geq 1$ the $d^{\text{th}}$ slices of the problems {\sf p-unbounded-change-approx}, {\sf p-bounded-change-approx} and {\sf p-0-1-change-approx} are {\sf NP}-complete. Moreover, these three problems are in {\sf para-NP}.
\end{theorem}

\begin{proof}
Membership of all three problems in \textsf{NP} is straightforward. We guess numbers $x_i \in [0,\lceil \frac{c}{c_i} \rceil]$ (for {\sf p-unbounded-change-approx}),
$x_i \in [0,b_i]$ (for {\sf p-bounded-change-approx}) and $x_i \in \{0,1\}$ (for {\sf p-0-1-change-approx}), respectively, and verify
\begin{displaymath}
\sum_{i=1}^m x_ic_i \geq c 
\quad \text{ and } \quad
a\left(\sum_{i=1}^m x_ic_i - c\right) \leq d .
\end{displaymath}
The same algorithm shows membership in \textsf{para-NP} for the parameterized problems.

Now we show \textsf{NP}-hardness. The proof uses a modification of the construction given in the appendix of \cite{goebbels}. We start with an instance of the \textsf{NP}-complete problem \textsc{subset sum} where we are given positive integers $a_1 \leq a_2 \leq \cdots \leq a_m$ and $a \in \naturals$ in binary encoding. The question is whether there
exist $x_1,\ldots,x_m \in \{0,1\}$ such that $\sum_{i=1}^m x_ia_i = a$.
W.l.o.g.~we can assume $a \leq ma_m < 2^{\#(ma_m)}$. Otherwise the instance is negative and the reduction yields a fixed negative instance. Furthermore we can assume w.l.o.g.~$a_i \neq 0$ for all $i \in [1,m]$. If $a > 0$ we can remove from the list all numbers that are equal to $0$. In the case $a=0$ the instance is trivially positive by choosing $x_i=0$ for all variables and the reduction yields a fixed positive instance. We define for $i \in [1,m]$ the numbers
\begin{eqnarray*}
c_{2i-1}&=& 0 + 2^{\#(ma_m)+(i-1)\#(m)} + 2^{\#(ma_m)+m\#(m)},\\
c_{2i}&=& a_i + 2^{\#(ma_m)+(i-1)\#(m)} + 2^{\#(ma_m)+m\#(m)}, \text{ and } \\
c &=& a + \sum_{i=1}^m 2^{\#(ma_m)+(i-1)\#(m)} + m \cdot 2^{\#(ma_m)+m\#(m)} .
\end{eqnarray*}
Note that we have $c_1 < c_2 < \dots < c_{2m}$. 

\begin{newclaim} \label{another claim}
If $\sum_{i=1}^m x_ia_i = a$ with $x_i \in \{0,1\}$ then we have 
\[\sum_{i=1}^m (x_ic_{2i} + (1-x_i)c_{2i-1}) =c.\]
\end{newclaim}

\noindent
\emph{Proof of Claim~\ref{another claim}.}
We have
\begin{align*}
\sum_{i=1}^m (x_ic_{2i} + (1-x_i)c_{2i-1})&= \sum_{i=1}^m x_ia_i + \sum_{i=1}^m 2^{\#(ma_m)+(i-1)\#(m)} + \sum_{i=1}^m 2^{\#(ma_m)+m\#(m)}\\
&= a + \sum_{i=1}^m 2^{\#(ma_m)+(i-1)\#(m)} + m \cdot 2^{\#(ma_m)+m\#(m)}\\
&=c,
\end{align*}
which proves the claim.
\qed

\begin{newclaim} \label{claim3}
Assume that $\sum_{i=1}^{2m} x_ic_i = c$ for $x_1, \ldots, x_{2m} \in \naturals$. Then the following holds: \bigskip
\begin{enumerate}
\item $\sum_{i=1}^{2m} x_i \leq m$ \bigskip
\item For $i \in [1,m]$ we have $x_{2i-1},x_{2i} \in \{0,1\}$ and $x_{2i-1} = 1$ if and only if $x_{2i}=0$. \bigskip
\item $\sum_{i=1}^m x_{2i}a_i = a$
\end{enumerate}
\end{newclaim}

\noindent
\emph{Proof of Claim~\ref{claim3}.}
For the first statement, suppose that $\sum_{i=1}^{2m} x_i > m$.
Then we obtain
\begin{eqnarray*}
\sum_{i=1}^{2m} x_ic_i &\geq& (m+1) \cdot c_1\\
&=& (m+1) \cdot (2^{\#(ma_m)} + 2^{\#(ma_m)+m\#(m)})\\
&\geq& 2^{\#(ma_m)} + (m+1) \cdot 2^{\#(ma_m)+m\#(m)}\\
&>& a + (m+1) \cdot 2^{\#(ma_m)+m\#(m)}\\
&=& a + m \cdot 2^{\#(ma_m)+m\#(m)} + 2^{\#(ma_m)} \cdot 2^{m\#(m)}\\
&>& a + m \cdot 2^{\#(ma_m)+m\#(m)} + 2^{\#(ma_m)} \cdot \sum_{i=1}^m \big(2^{\#(m)}\big)^{i-1}\\
&=& a + m \cdot 2^{\#(ma_m)+m\#(m)} + \sum_{i=1}^m 2^{\#(ma_m)+(i-1)\#(m)} \ = \ c ,
\end{eqnarray*}
which is a contradiction.

Let us now show the second statement of Claim~\ref{claim3}. By the first statement 
we have $\sum_{i=1}^{2m} x_i \leq m$. Hence, the sum $c = \sum_{i=1}^{2m} x_ic_i$ is a 
sum of $m$ numbers $c_i$, possibly with repetitions. The binary representations of the $c_i$ are divided into blocks. The first block contains the $\#(ma_m)$ first bits and the next blocks consist of exactly $\#(m)$ bits each. The lengths of the blocks are such that in the sum $\sum_{i=1}^{2m} x_ic_i$ a carry cannot enter a block from the previous block.
The bit at position $1 + \#(ma_m)+(i-1)\#(m)$ (i.e., the first bit of the $i^{\text{th}}$ block) in $c, c_{2i-1}, c_{2i}$ is $1$ (due to the summand $2^{\#(ma_m)+(i-1)\#(m)}$ in these numbers). 
In all other numbers $c_j$, the bit at position $1 + \#(ma_m)+(i-1)\#(m)$ is zero and by adding at most $m$ numbers $c_i$ one cannot get a carry at position $1 + \#(ma_m)+(i-1)\#(m)$.
This implies that $x_{2i-1},x_{2i} \in \{0,1\}$ and $x_{2i-1} = 1$ if and only if $x_{2i}=0$.

Finally, the third statement of Claim~\ref{claim3} follows from 
 the identity $c = \sum_{i=1}^{2m} x_ic_i$ (where $\sum_{i=1}^{2m} x_i \leq m$ by the first statement) and only taking the first block consisting of the $\#(ma_m)$
 low-order bits.
 \qed
 
 \medskip
 \noindent
Claim~\ref{claim3} gives us a reduction from the
subset sum equation
\begin{displaymath}
\sum_{i=1}^m x_ia_i = a
\end{displaymath}
to an equation
\begin{displaymath}
\sum_{i=1}^{2m} x_ic_i = c
\end{displaymath}
with $c_1 < \dots < c_{2m}$. Moreover since by Claim~\ref{claim3} (second statement) it is ensured that we have for all variables $x_i \in \{0,1\}$ this reduction works in both the bounded and unbounded setting. 

Now we come to the actual reduction from \textsc{subset sum} to the $d^{\text{th}}$ slice of {\sf p-unbounded-change-approx} (respectively, {\sf p-bounded-change-approx} and {\sf p-0-1-change-approx}).
We have $\sum_{i=1}^{2m} x_ic_i = c$ if and only if $\sum_{i=1}^{2m} x_i(d+1)c_i = (d+1)c$, which holds if and only if the following two conditions hold:
\begin{eqnarray*}
\sum_{i=1}^{2m} x_i (d+1)c_i & \geq & (d+1)c , \\
a \left(\sum_{i=1}^m x_i(d+1)c_i -(d+1)c\right) & \leq & d .
\end{eqnarray*}
Note that $a \geq 1$. This concludes the reduction.
\end{proof}

\begin{theorem}
When restricted to $a \geq 1$ and $b=0$, 
 {\sf p-unbounded-change-approx}, {\sf p-bounded-change-approx} and {\sf p-0-1-change-approx} are {\sf para-NP}-complete.
\end{theorem}
\begin{proof}
By \cite[Theorem 2.14]{flumgrohe} a nontrivial parameterized problem in \textsf{para-NP} is \textsf{para-NP}-complete under fpt-reductions if and only if the union of finitely many slices is \textsf{NP}-complete under polynomial time reductions (the $d^{\text{th}}$ slice of a parameterized problem is the restriction of the problem, where the parameter is set to the fixed value $d$). 
By Theorem~\ref{np-completeslices} already a single slice of the problems from the theorem is \textsf{NP}-complete under polynomial time reductions.
\end{proof}

\section{Conclusion}
In this paper we have shown the equivalence of $\sgfactor$, $\sgks$ and $\sgsss$ with respect to fpt-reductions. This may provide a tool for showing new upper and lower bounds for $\sgfactor$ in the \textsf{W}-hierarchy because the equivalence resolves the problem to ensure the generators to factor in a specific order. The equivalence allows now to give the generators in a suitable order.

Moreover we have shown equivalence with respect to fpt-reductions of several problems, namely subset sum problems (for integers), change making problems and the special case of permutation group factorization for cyclic permutation groups ($\cpgfactor$). By this and the membership in $\mathsf{W}[3]$ for the subset sum problem \cite{bussislam} we obtain membership in $\mathsf{W}[3]$ of the change making problems which addresses the gap between $\mathsf{W}[1]$-hardness and membership in $\mathsf{XP}$ that was obtained in \cite{goebbels}. In the same paper it was also asked for better lower bounds. However, improving the $\mathsf{W}[1]$-hardness 
 for any of the change making problems seems to be very challenging for several reasons.
 \begin{itemize}
 \item 
Because of the equivalence with respect to fpt-reductions, a better lower bound for the change making problems would directly
yield a better lower bound for permutation group factorization, for which in almost 30 years since the appearance of the paper \cite{caichendowneyfellows} nothing better than $\mathsf{W}[1]$-hardness has been shown.
\item Proving $\mathsf{W}[3]$-completeness of the change making problems (and hence the  problems from Theorem~\ref{maintheorem}) would be surprising since it would imply
\begin{displaymath}
\textsf{W}[3] \subseteq \textsf{W}^{\text{func}}[2] \subseteq \textsf{A}[2],
\end{displaymath}
because permutation group factorization is a special case of transformation monoid factorization ($\tmfactor$) which is contained in $\textsf{W}^{\text{func}}[2]$ by Theorem~\ref{thm-tmfactor}. By \cite{flumgrohe} we only know
\begin{displaymath}
\textsf{W}[2] \subseteq \textsf{W}^{\text{func}}[2] \subseteq \textsf{A}[2].
\end{displaymath}
\item The previous point still leaves room for an $\mathsf{W}[2]$-hardness proof. But there has 
been some progress in showing $\mathsf{W}[1]$-completeness for subset sum over $\integers$ \cite{abboud}, which makes it unlikely that better lower bounds are achievable. In \cite{abboud} the following variant of subset sum over $\integers$ has been studied:
The input  consists of integers $z_1,\dots,z_m \in \integers$ and $k \in \naturals$ (the parameter) and it is asked whether 
there are numbers $x_1,\dots,x_m \in \{0,1\}$ such that
\begin{displaymath}
\sum_{i=1}^m x_iz_i = 0 \text{ with } \sum_{i=1}^m x_i = k.
\end{displaymath}
It is shown in \cite{abboud} that this problem is $\mathsf{W}[1]$-complete when the numbers $z_1,\dots,z_m$ are restricted to $[-n^{2k},n^{2k}]$.
In \cite{abboud} the result was also 
improved to $\mathsf{W}[1]$-completeness for unrestricted integers $z_1,\dots,z_m$
under the assumption that a certain circuit lower bound holds.

Note that the subset sum version from \cite{abboud} is a slight variant of our subset sum version $\Zsss$, 
where 
the input consists of integers $z_1,\dots,z_m,z \in \integers$ and $k \in \naturals$ (the parameter) and it is asked whether there are numbers $x_1,\dots,x_m \in \{0,1\}$ such that
\begin{equation} \label{our-subsetsum}
\sum_{i=1}^m x_iz_i = z \text{ with } \sum_{i=1}^m x_i = k.
\end{equation}
But for $k \geq 1$ this is equivalent to asking whether there are numbers $x_1,\dots,x_m \in \{0,1\}$ such that
\begin{displaymath}
\sum_{i=1}^m x_i(kz_i-z) = 0 \text{ with } \sum_{i=1}^m x_i = k
\end{displaymath}
In the case $k = 0$, \eqref{our-subsetsum} holds if and only if $z = 0$. 
Hence, the subset sum variant from  \cite{abboud} is equivalent to  $\Zsss$ 
with respect to fpt-reductions.
Under the circuit lower bound assumption from \cite{abboud} also all problems mentioned in Theorem~\ref{maintheorem} are $\mathsf{W}[1]$-complete since all these problems are equivalent to $\Zsss$ with respect to fpt-reductions.
\end{itemize}
It would be also interesting to study the parameterized complexity
of $\tmfactor$ for restricted classes of transformation monoids (e.g. aperiodic or commutative transformation monoids).
Due to the fpt-equivalence of $\tmfactor$ and $\mathsf{p\text{-}DFA\text{-}SW}$, this question should be
related to the parameterized complexity of $\mathsf{p\text{-}DFA\text{-}SW}$ (Problem~\ref{pdfasw}) for 
DFAs where the transformation monoid of the DFA is from a restricted class.
For the non-parameterized version of $\mathsf{p\text{-}DFA\text{-}SW}$, several {\sf NP}-completeness
results for some algebraic classes of transformation monoids (aperiodic monoids and commutative monoids among others) have been
shown in \cite{Martyugin09}.

A generalization of the parameterized factorization problem $\Cfactor$
is the \emph{parameterized rational subset membership problem} for the class of monoids $\mathsf{C}$. The input consists of a monoid $M \in \mathsf{C}$, a finite (nondeterministic) automaton $\mathcal{A}$, whose transitions are
labelled with elements from $M$ and an additional element $a \in M$. The automaton $\mathcal{A}$ defines a subset $L(\mathcal{A}) \subseteq M$ in the natural way
(take the set of all paths in $\mathcal{A}$ from an initial state to a final state and for each path multiply the $M$-labels of the transitions in the 
order given by the path) and the question is whether $a \in L(\mathcal{A})$. For the parameter we take the number of states of 
$\mathcal{A}$.
Clearly, all lower bounds for $\Cfactor$ (for every class $\mathsf{C}$) are inherited by the 
parameterized rational subset membership problem for $\mathsf{C}$.
One might therefore investigate whether better lower bounds can be shown for the 
parameterized rational subset membership problem or whether there is an fpt-reduction from 
the parameterized rational subset membership problem for a class $\mathsf{C}$ to $\Cfactor$.
The non-parameterized rational subset membership problem for permutation groups was shown to 
be $\mathsf{NP}$-complete in \cite{Khashaev22,LohreyRZ22}.

\bibliographystyle{abbrvnat}
\bibliography{bib}

\appendix

\section{Additional material on cyclic permutation groups} \label{appendix-cyclic}

\begin{lemma}\label{lemmacyclicgroup}
There is an algorithm that given two permutations $\alpha,\beta \in S_n$ checks in polynomial time whether $\langle \alpha,\beta \rangle$ is cyclic and if so returns a permutation $\sigma$ such that $\langle \alpha,\beta \rangle = \langle \sigma \rangle$ and 
\[\lcm(\ord(\alpha),\ord(\beta)) = \ord(\sigma).\]
\end{lemma}

\begin{proof}
On input $\alpha,\beta \in S_n$, the algorithm computes the following data:
\begin{enumerate}
\item  $\ord(\alpha)$ and $\ord(\beta)$,
\item the prime factorizations $\ord(\alpha) = p_1^{a_1} \cdots p_m^{a_m}$ and $\ord(\beta) = p_1^{b_1} \cdots p_m^{b_m}$
(we assume here that $a_i$ or $b_i$ is non-zero for every $i \in [1,m]$),
\item the numbers $r = \prod_{i=1}^m p_i^{x_i}$ and $s = \prod_{i=1}^m p_i^{y_i}$, where
for all $i \in [1,m]$ we have
\[
x_i = \begin{cases}0 & \text{if } a_i \geq b_i\\ a_i  & \text{if } a_i < b_i\end{cases} \text{ and }
y_i = \begin{cases}0 & \text{if } b_i > a_i\\ b_i  & \text{if } b_i \leq a_i\end{cases}
\]
\item the permutations $\alpha' = \alpha^r$, $\beta' = \beta^s$, and $\gamma = \alpha'\beta'$.
\end{enumerate}
If $\langle \alpha,\beta \rangle = \langle \gamma \rangle$ then the algorithm returns $\gamma$, otherwise it returns ``not cyclic''.

We first show that the algorithm is correct, then we argue that it runs in polynomial time.
Clearly, if $\langle \alpha,\beta \rangle$ is not cyclic then the algorithm will return ``not cyclic''.
Vice versa suppose there is a permutation $\sigma$ such that $\langle \alpha,\beta \rangle = \langle \sigma \rangle$. Then we have $\ord(\sigma) = |\langle \alpha,\beta \rangle| = \lcm(\ord(\alpha),\ord(\beta))$, where $\lcm$ denotes the least common multiple. Now let $\alpha'$, $\beta'$ and $\gamma$ be as defined in the algorithm. 
Then $\ord(\alpha') = \ord(\alpha)/r$ and $\ord(\beta') = \ord(\beta)/s$ and these numbers  
are coprime. We obtain $\ord(\gamma) = \ord(\alpha')\ord(\beta') = \lcm(\ord(\alpha),\ord(\beta)) = \ord(\sigma)$. Moreover $\gamma \in \langle \alpha,\beta \rangle = \langle \sigma \rangle$. By this it follows that $\gamma$ is also a generator of $\langle \sigma \rangle$, i.e., $\langle \gamma \rangle = \langle \alpha,\beta \rangle$.
Hence, checking $\langle \alpha,\beta \rangle = \langle \gamma \rangle$ succeeds and the algorithm correctly returns $\gamma$.

We now argue that the algorithm runs in polynomial time. Recall that $n$ (the degree of the permutations) is given in unary notation.
To compute the order of a permutation we simply have to compute the length of the disjoint cycles and compute the least common multiple of the lengths which can be computed in polynomial time. By Lagrange's theorem both $\ord(\alpha)$ and $\ord(\beta)$ divide $n!$ and thus $p_i \leq n$ for $i \in [1,m]$. Moreover $m \leq n$ because every prime $p_i$ is bounded by $n$. Furthermore the exponents $a_i$ and $b_i$ are bounded by $\mathcal{O}(\log n)$ because every cycle has length at most $n$. Hence the prime factorizations of $\ord(\alpha)$ and $\ord(\beta)$ can be computed in polynomial time. Clearly $\alpha',\beta'$ and $\gamma$ can be computed in polynomial time. Finally, in order to check $\langle \alpha,\beta \rangle = \langle \gamma \rangle$, it suffices to check whether $\alpha \in \langle \gamma \rangle$ and $\beta \in \langle \gamma \rangle$.
This can be done in polynomial time by  \cite{FurstHL80}.
\end{proof}

\begin{lemma}\label{lemmacyclicgroup2}
There is an algorithm that given a list of permutations $\pi_1, \ldots, \pi_m \in S_n$ checks in polynomial time whether $\langle \pi_1, \ldots, \pi_m \rangle$ is cyclic and if so returns a permutation $\sigma$ such that $\langle \pi_1, \ldots, \pi_m \rangle = \langle \sigma \rangle$.
\end{lemma}

\begin{proof}
If $\langle \pi_1, \ldots, \pi_m \rangle$ is cyclic then also $\langle \pi_1, \ldots, \pi_l \rangle$ is cyclic for every $l \leq m$.
Therefore, using Lemma~\ref{lemmacyclicgroup} we can first check whether $\langle \pi_1, \pi_2 \rangle$ is cyclic. If this is not the case, then $\langle \pi_1, \ldots, \pi_m \rangle$
is not cyclic. Otherwise, Lemma~\ref{lemmacyclicgroup} returns a permutation $\pi_{1,2}$ such that $\langle \pi_1, \pi_2 \rangle = \langle \pi_{1,2} \rangle$.
We then recursively check whether $\langle \pi_{1,2}, \pi_3, \ldots, \pi_m \rangle$ is cyclic.
\end{proof}

\end{document}